\documentclass[10pt]{article}
\usepackage{color}
\usepackage{amsmath}
\usepackage{graphicx}
\usepackage{footnote}
\usepackage{subscript}
\usepackage{amsfonts,color}
\usepackage{amsmath,amssymb}
\usepackage{mathtools}
\usepackage{amssymb} 
\usepackage{mathtools}
\usepackage{amsthm,wasysym}
\usepackage[english]{babel}
\usepackage{cancel}
\usepackage{colonequals}
\makeatletter
\usepackage{float}
\usepackage{nicefrac,enumitem}
\usepackage{hyperref}
\usepackage{fixltx2e}

\makeatletter
\let\@fnsymbol\@arabic
\makeatother

\newcommand{\UEN}{\overline{U}\hspace{-0.07cm}_e\hspace{-0.08cm}^{\natural}}

\usepackage{bbm}
\newcommand{\id}{{ {\mathbbm{1}}}}

\newcommand{\tr}{{\rm tr}}
\newcommand{\dev}{{\rm dev}}
\newcommand{\sym}{{\rm sym}}
\newcommand{\skw}{{\rm skew}}

\newcommand{\Curl}{{\rm Curl}}

\newcommand{\norm}[1]{\|#1\|}

\def\dd{\displaystyle}

\newtheorem{theorem}{Theorem}[section]

\newtheorem{proposition}[theorem]{Proposition}

\def\dd{\displaystyle}

\usepackage{multicol}
\usepackage[font=small]{caption}
\newcommand{\citet}[2][]{\citeauthor{#2} \cite[#1]{#2}}
\graphicspath{{gfx/}}
\RequirePackage{amsmath}	
\RequirePackage{mathtools}	
\newcommand{\col}{\colon}

\newcommand{\R}{\mathbb{R}}

\DeclareMathOperator{\SO}{SO}

\DeclareMathOperator{\so}{\mathfrak{so}}

\renewcommand{\skew}{\mathop{\mathrm{skew}}\nolimits}

\DeclareMathOperator{\anti}{anti}
\DeclareMathOperator{\axl}{axl}

\newcommand{\pdd}[3][]{\frac{\partial\ifx&#1&\else^{#1}\fi #2}{\partial #3}}

\DeclareMathOperator{\@macros@div}{div}
\renewcommand{\div}{\@macros@div}

\newcommand{\D}{\mathrm{D}}
\newcommand{\dynnorm}[1]{\left\lVert #1 \right\rVert}
\newcommand{\innerproduct}[1]{\langle #1 \rangle}

\newcommand{\iprod}{\innerproduct}
\newcommand{\dyniprod}[1]{\left\langle #1 \right\rangle}

\newcommand{\matr}[1]{\begin{pmatrix} #1 \end{pmatrix}}

\providecommand{\availableaturl}[2][]{%
	available at \url{#2}%
}

\usepackage[babel,german=quotes]{csquotes} 

\makeatletter%
\@ifpackageloaded{hyperref}{}{}%
\makeatother%

\usepackage{color}
\usepackage{amsmath}

\usepackage{footnote}
\usepackage{subscript}
\usepackage{amsfonts,color}
\usepackage{amsmath,amssymb}
\usepackage{amssymb} 
\usepackage{mathtools}
\usepackage{amsthm,wasysym}
\usepackage{framed}
\usepackage{cancel}
\usepackage{graphicx}
\usepackage{caption}
\usepackage{subcaption}
\usepackage{rotating}

\usepackage[english]{babel}
\usepackage[utf8]{inputenc}
\makeatletter
\usepackage{float}
\usepackage{wrapfig}
\restylefloat{figure}

\makeatletter
\let\@fnsymbol\@arabic
\makeatother

\usepackage{bbm}

\def\dd{\displaystyle}

\def\barr{\begin{array}}
	\usepackage{lscape}

	\def\earr{\end{array}}
\def\bec#1{\begin{equation}\label{#1}}
\def\becn{\begin{equation*}}
\def\endec{\end{equation}}
\def\endecn{\end{equation*}}
\def\dd{\displaystyle}
\def\bfm#1{\mbox{\boldmat}}

\renewcommand{\dd}{\displaystyle}

\usepackage{enumitem}

\usepackage{pifont}

\newcommand{\imin}[1][\lambda]{m{\ifx&#1&\else(#1)\fi}}
\newcommand{\imax}[1][\lambda]{M{\ifx&#1&\else(#1)\fi}}
\newcommand{\iset}[1][\lambda]{J{\ifx&#1&\else(#1)\fi}}
\allowdisplaybreaks[1]

\usepackage{titletoc}

\begin{document}

\title{Explicit formula for  the Gamma-convergence homogenized quadratic  curvature energy in isotropic Cosserat  shell models}
\author{  
	Maryam Mohammadi Saem\thanks{Maryam Mohammadi Saem,  \ \  Lehrstuhl f\"{u}r Nichtlineare Analysis und Modellierung, Fakult\"{a}t f\"{u}r
		Mathematik, Universit\"{a}t Duisburg-Essen,  Thea-Leymann Str. 9, 45127 Essen, Germany, email: maryam.mohammadi-saem@stud.uni-due.de} \quad and  \quad  Emilian Bulgariu\thanks{ Emilian Bulgariu,\ Department of Exact Sciences, Faculty of Horticulture,   Ia\c si University of Life Sciences (IULS),  Aleea Mihail Sadoveanu no.3, 700490 Ia\c si,
		Romania, email:  ebulgariu@uaiasi.ro} \quad and\quad \\ Ionel-Dumitrel Ghiba\thanks{Ionel-Dumitrel Ghiba, \ \  Alexandru Ioan Cuza University of Ia\c si, Department of Mathematics,  Blvd.
		Carol I, no. 11, 700506 Ia\c si,
		Romania; and  Octav Mayer Institute of Mathematics of the
		Romanian Academy, Ia\c si Branch,  700505 Ia\c si, email: dumitrel.ghiba@uni-due.de, dumitrel.ghiba@uaic.ro}\quad and\quad 	
		    Patrizio Neff\,\thanks{Patrizio Neff,  \ \ Head of Lehrstuhl f\"{u}r Nichtlineare Analysis und Modellierung, Fakult\"{a}t f\"{u}r
		Mathematik, Universit\"{a}t Duisburg-Essen,  Thea-Leymann Str. 9, 45127 Essen, Germany, email: patrizio.neff@uni-due.de}
}

\maketitle
\begin{abstract}
	We show how to explicitly compute the homogenized curvature energy appearing in the isotropic $\Gamma$-limit for flat and for curved initial configuration Cosserat shell models, when a parental three-dimensional minimization problem on $\Omega \subset \mathbb{R}^3$ for a Cosserat energy  based on the second order dislocation density tensor $\alpha:=\overline{R} ^T \Curl\overline{R} \in \mathbb{R}^{3\times 3}$, $\overline{R}\in {\rm SO}(3)$  is used. 
\end{abstract}
\tableofcontents
\section{Introduction}\label{intro}
  The Cosserat theory  introduced  by the Cosserat brothers in 1909 \cite{Cosserat09,Cosserat08} represents a generalization of the elasticity theory. While the elasticity theory models each constituent particle of the body as a material point, i.e., it is able to model only the translation of each particle through the classical deformation $\varphi\col \Omega\subset \R^3\to \R^3$,  the Cosserat theory models the micro-rotation of each particle attaching to each material point an independent triad of orthogonal directors, the microrotation $ {\overline{R}}\col \Omega\subset \R^3\to \SO(3)$.  Invariance of the energy under superposed rigid body motions (left-invariance under $\SO(3)$) allowed them to conclude the suitable form of the energy density as  $W=W( \overline{U}, {\mathfrak{K}})$, where $\overline{U}:={\overline{R}}^T\D \varphi$ is the first Cosserat deformation tensor and $ {\mathfrak{K}}:=({\overline{R}^T\partial_{x_1} \overline{R}}, {\overline{R}^T\partial_{x_2} \overline{R}}, {\overline{R}^T\partial_{x_3} \overline{R}})$ is the second Cosserat deformation tensor.  The Cosserat brothers never considered specific forms of the elastic energy and they never linearised their model to obtain the well-known linear Cosserat (micropolar) model \cite{ghibarizzi}. In the present paper we  only consider isotropic material, i.e., the behaviour of the elastic material is modelled with the help of an additionally right-invariant  energy under $\SO(3)$. In addition we will  consider  quadratic energies in suitable strains (a physically linear dependence of the stress tensor and of the couple-stress tensor on the strain measures)  which allows an explicit and  practical \cite{Nebel,Sander} representation of the energy.
 
  In \cite{Maryam} we have provided  a nonlinear membrane-like Cosserat shell model on a curved reference configuration starting from a geometrically nonlinear, physically linear three-dimensional isotropic Cosserat model. Beside the change of metric, the obtained membrane-like Cosserat shell model \cite{Maryam} is still capable to capture the transverse shear deformation and the  {Cosserat}-curvature due to remaining Cosserat effects.  The Cosserat-shell model presented in \cite{Maryam}  for curved initial configuration generalizes the Cosserat-shell model constructed in \cite{neff2007geometrically1} for flat initial configurations.
 There are many different ways to mathematically model shells \cite{Koiter69}, e.g., the \textit{derivation approach} \cite{neff2004geometricallyhabil,neff2004geometrically,GhibaNeffPartI,GhibaNeffPartII,GhibaNeffPartIV,GhibaNeffPartIII,GhibaNeffPartV}, the \textit{intrinsic approach}  \cite{Altenbach-Erem-11,Altenbach-Erem-Review,green1965general}, the \textit{asymptotic method}, the \textit{direct approach} \cite{Naghdi69,antman2005nonlinear, cohen1966nonlinear1, cohen1966nonlinear, cohen1989mathematical, Cosserat09, ericksen1957exact, green1965general, rubin2013cosserat,birsan2019refined,boyer2017poincare,birsan2020derivation,birsan2023}). However, {\it Gamma-convergence} methods are  preferred in the mathematical community. 
 
 When the Cosserat parental  three-dimensional energy is considered, in the deduction of the Gamma-limit for the curved initial configuration  we have to construct four homogenized energies, while in the expression of the Gamma-limit will appear only two: the homogenized  membrane energy and the homogenized curvature energy. In the deduction of the Gamma-limit in \cite{Maryam},  we have explicitly stated the form of the homogenized membrane energy, while  the explicit form of the homogenized curvature energy was only announced and we have used only some implicit properties (its continuity, convexity, etc.). The same was done in the deduction of the Gamma-limit for a flat initial configuration in \cite{neff2007geometrically1} (we notice that another form of the Cosserat-curvature energy was considered) and no explicit form of the  homogenized curvature energy could be given. In \cite{Maryam} we have announced the form of the homogenized curvature energy without giving details about its deduction. Therefore, this is the main aim of this paper, i.e., to provide the solutions of all optimization problems needed for having an explicit approach of Cosserat shell model for  flat (Section \ref{flat}) and curved (Section \ref{curved}) initial configurations via the Gamma-limit. The second goal is to point out the advantages, at least from a computation point of view, of the usage of the curvature strain tensor $\alpha:=\overline{R}^T \Curl \overline{R}$ in the parental three-dimensional Cosserat-curvature energy, instead of other curvature strain tensors considered in literature.
 We mention that, even if $\alpha$ is   controlled by $\widehat{\mathfrak{K}} :=\Big(\overline{R} ^T\D (\overline{R} .e_1),R^T\D (\overline{R} .e_2),R^T  \D (\overline{R} .e_3)\Big)\in \R^{3\times 9}$ used in \cite{neff2004geometrically,neff2007geometrically1},  an explicit expression  via Gamma-convergence  of the homogenisation of the quadratic  curvature energy in terms of the third order  tensor $\widehat{\mathfrak{K}} $ is missing in the literature, even for flat initial configuration. In fact it turned out that  $\widehat{\mathfrak{K}}$ is frame-indifferent but not itself isotropic, a fact which makes it unsuitable to be used in an isotropic model. Beside these advantages  of the usage of the  second order dislocation density tensor $\alpha$, there is a one-two-one relation between $\alpha$ and  the so-called wryness tensor  $\Gamma := \Big(\mathrm{axl}(\overline{R} ^T\,\partial_{x_1} \overline{R} )\,|\, \mathrm{axl}(\overline{R} ^T\,\partial_{x_2} \overline{R} )\,|\,\mathrm{axl}(\overline{R} ^T\,\partial_{x_3} \overline{R} )\,\Big)\in \mathbb{R}^{3\times 3}$ (second order Cosserat deformation tensor \cite{Cosserat09},
 a Lagrangian strain measure for curvature-orientation change\cite{Pietraszkiewicz04}). This property  is not shared with $\widehat{\mathfrak{K}} $ and  $\mathfrak{K}$. We show that considering $\mathfrak{K} $ is equivalent to a particular choice of the constitutive coefficients  in a model in which $\alpha$ is used and, therefore, the formulas determined for the homogenized quadratic  curvature energy via Gamma convergence are valid for  parental isotropic three-dimensional energies which are quadratic in $\mathfrak{K} $, too. However, the general form of a quadratic isotropic energy has a more complicated form in terms of $\mathfrak{K} $ in comparison to the case when we express it in terms of $\alpha$, see Subsection \ref{conections}. Therefore, from a computational point of view it is more convenient to consider $\alpha$ in a isotropic Cosserat model.
Moreover, using \cite{Neff_curl08} (see Subsection \ref{conections}),
 we have that $\alpha$ controls $\mathfrak{K}$ in $L^2(\Omega, \R^{3\times 3\times 3})$ which controls  $\widehat{\mathfrak{K}}$ in $L^2(\Omega, \R^{3\times 3\times 3})$ which controls  $\alpha$ in $L^2(\Omega, \R^{3\times 3})$. 
 Therefore, a positive definite quadratic form in terms of one of the three variants of appropriate curvature tensors is energetically controlled by each of the three Cosserat-curvature tensors. This is why when the problem of existence of the solution or other similar qualitative results are considered, the form of the used Cosserat-curvature strain tensor is irrelevant, in the sense that if such a result is obtained for a  Cosserat-curvature quadratic in Cosserat-curvature strain tensor then it may be immediately extended for   Cosserat-curvature quadratic energies in the other two Cosserat-curvature strain tensor considered in the present paper.
 However, the usage of the Cosserat-curvature strain tensor $\alpha$ has two main advantages: 
 \begin{itemize}
 	\item 
 	a quadratic isotropic energy in terms of the wryness tensor $\Gamma$ is rewritten in a transparent and explicit form as a quadratic energy in terms of dislocation density tensor $\alpha$ (and vice versa), see Subsection \ref{conections};
 		\item the expression of 
 	a quadratic isotropic energy in terms of the wryness tensor  $\Gamma$ is very simple and suitable for analytical computations, see Subsection \ref{conections};
 	\item it admits the explicit analytical calculation of the homogenized quadratic curvature energies in the construction of the Cosserat shell model via $\Gamma$-convergence method, see Sections \ref{flat} and \ref{curved}.
 \end{itemize}
\section{Three dimensional geometrical nonlinear and physical linear\\ Cosserat models}
\subsection{General notation}
	Before continuing, let us introduce the notation we will use or have already used in Section \ref{intro} and in the abstract.  We denote by $\mathbb{R}^{m\times n}$, $n,m\in\mathbb{N}$, the set of real $m\times n$ second order tensors, written with
	capital letters. We adopt the usual abbreviations of Lie-group theory, i.e.,
	${\rm GL}(n)=\{X\in\mathbb{R}^{n\times n}\;|\det({X})\neq 0\}$ the general linear group,
	${\rm SL}(n)=\{X\in {\rm GL}(n)\;|\det({X})=1\},\;
	\mathrm{O}(n)=\{X\in {\rm GL}(n)\;|\;X^TX=\id_n\},\;{\rm SO}(n)=\{X\in {\rm GL}(n)| X^TX=\id_n,\det({X})=1\}$ with
	corresponding Lie-algebras $\mathfrak{so}(n)=\{X\in\mathbb{R}^{n\times n}\;|X^T=-X\}$ of skew symmetric tensors
	and $\mathfrak{sl}(n)=\{X\in\mathbb{R}^{n\times n}\;| \tr({X})=0\}$ of traceless tensors. Here, 
	for $a,b\in\mathbb{R}^n$ we let $ \bigl\langle {a},{b}\bigr\rangle _{\mathbb{R}^n}$  denote the scalar product on $\mathbb{R}^n$ with
	associated
	{(squared)}
	vector norm $\lVert a\rVert_{\mathbb{R}^n}^2= \bigl\langle {a},{a}\bigr\rangle _{\mathbb{R}^n}$.
	The standard Euclidean scalar product on $\mathbb{R}^{n\times n}$ is given by
	$ \bigl\langle  {X},{Y}\bigr\rangle _{\mathbb{R}^{n\times n}}={\rm tr}(X Y^T)$, and thus the
	{(squared)}
	Frobenius tensor norm is
	$\lVert {X}\rVert^2= \bigl\langle  {X},{X}\bigr\rangle _{\mathbb{R}^{n\times n}}$. In the following we omit the index
	$\mathbb{R}^n,\mathbb{R}^{n\times n}$. The identity tensor on $\mathbb{R}^{n \times n}$ will be denoted by $\id_n$, so that
	${\rm tr}({X})= \bigl\langle {X},{\id}_n\bigr\rangle $. We let ${\rm Sym}(n)$ and ${\rm Sym}^+(n)$ denote the symmetric and positive definite symmetric tensors, respectively.  For all $X\in\mathbb{R}^{3\times3}$ we set ${\rm sym}\, X=\frac{1}{2}(X^T+X)\in{\rm Sym}(3)$, $\skw\, X\,=\frac{1}{2}(X-X^T)\in \mathfrak{so}(3)$ and the deviatoric part $\dev  \, X=X-\frac{1}{n}\;\tr(X)\,\id_n\in \mathfrak{sl}(n)$  and we have
	the {orthogonal Cartan-decomposition  of the Lie-algebra} 
	$
	\mathfrak{gl}(3)=\{\mathfrak{sl}(3)\cap {\rm Sym}(3)\}\oplus\mathfrak{so}(3) \oplus\mathbb{R}\!\cdot\! \id_3,\ 
	X=\dev  \,\sym \,X+ \skw\,X+\frac{1}{3}\tr(X)\, \id_3\,.
	$
We use the canonical identification of $\mathbb{R}^3$ with $\so(3)$, and, for
$
{A}=	\begin{footnotesize}\begin{pmatrix}
0 &-a_3&a_2\\
a_3&0& -a_1\\
-a_2& a_1&0
\end{pmatrix}\end{footnotesize}\in \so(3)
$
we consider the operators ${\rm axl}\,:\so(3)\rightarrow\mathbb{R}^3$ and $\anti:\mathbb{R}^3\rightarrow \so(3)$ through
$
{\rm axl}\, {A}:=\left(
a_1,
a_2,
a_3
\right)^T,$ $  {A}.\, v=({\rm axl}\,  {A})\times v, $ $ ({\rm anti}(v))_{ij}=-\epsilon_{ijk}\,v_k \ \  \forall \, v\in\mathbb{R}^3,
$ $({\rm axl}\,  {A})_k=-\frac{1}{2}\, \epsilon_{ijk} {A}_{ij}=\frac{1}{2}\,\epsilon_{kij} {A}_{ji}\,,$ $  {A}_{ij}=-\epsilon_{ijk}\,({\rm axl}\,  \mathbf{A})_k=:{\rm anti}({\rm axl}\,  {A})_{ij},
$
where $\epsilon_{ijk}$ is the totally antisymmetric third order permutation tensor.
	For $X\in {\rm GL}(n)$,  ${\rm Adj}({X})$  denotes the tensor of
	transposed cofactors, while  the $(i, j)$ entry of the cofactor is the $(i, j)$-minor times a sign factor.  
	Here, given $z_1,z_2,z_3\in \mathbb{R}^{n\times k}$,  the notation $(z_1\,|\,z_2\,|\,z_3)$    means a matrix $Z\in \mathbb{R}^{n\times 3k}$ obtained by taking $z_1,z_2,z_3$ as block matrices. A third order tensor $A=(A_{ijk})\in \mathbb{R}^{3\times 3\times 3}$ will be replaced with an equivalent object, by reordering its  components in a $\mathbb{R}^{3\times 9}$ matrix
$
	A\equiv(A_1\,|\,A_2\,|\,A_3)\in \mathbb{R}^{3\times 9}, \  A_k:=(A_{ijk})_{ij}=A.\, e_k\in \mathbb{R}^{3\times 3}, \ k=1,2,3,
$
	and we consider
	$ \text{sym}A=\big(\text{sym}\,A_1\,\,|\,\text{sym}\,A_2,\,|\,\text{sym}\,A_3 \big)\in\mathbb{R}^{3\times 9},$ $
	\text{skew}A=\big(\text{skew}A_1\,|\,\text{skew}A_2\,|\,\text{skew}A_3 \big)\in\mathbb{R}^{3\times 9},$ $
	\tr (A)=\tr (A_1)+\tr (A_2)+\tr (A_3).
	$
	Moreover, we define the products of  a second order tensor $B=(B_{ij})_{ij}\in\mathbb{R}^{3\times 3}$ and a  third order tensor $A=(A_1\,|\,A_2\,|\,A_3)\in  \mathbb{R}^{3\times 9}$ in a natural way as 
$B\, A=(B\,A_1\,|\,B\,A_2\,|\,B\,A_3)\in  \mathbb{R}^{3\times 9},$ $
	A\,B=\big(\sum_{k=1}^3 A_k\,B_{k1}\,|\, \sum_{k=1}^3 A_k\,B_{k2}\, |\,\sum_{k=1}^3 A_k\,B_{k3}\big)\in  \mathbb{R}^{3\times 9}.
	$
	Let us remark that for  $B=(B_{ij})_{ij}\in {\rm GL}^+(3)$ having the inverse $B=(B^{ij})_{ij}$ and  $A,C\in\mathbb{R}^{3\times 9}$ the following equivalences hold true
	$
	A\, B=C\  \Leftrightarrow\  \sum_{k=1}^3 A_k\,B_{kl}=C_l\  \Leftrightarrow\  \sum_{l=1}^3(B^{lm}\sum_{k=1}^3 A_k\,B_{kl})=\sum_{l=1}^3C_l\,B^{lm}   \Leftrightarrow\  A=C\, B^{-1}.
	$
	We define the norm of a  third order tensor $A=(A_1\,|\,A_2\,|\,A_3)\in  \mathbb{R}^{3\times 9}$ by
	$
	\|A\|^2=\sum_{k=1}^3\|A_k\|^2.
	$
For $A_1, A_2, A_3\in \so(3)$ we define $\axl A=(\axl A_1\,|\,\axl A_2\,|\,\axl A_3)\in \mathbb{R}^{3\times 3}$, while for $z=(z_1\,|\,z_2\,|\,z_3)\in \mathbb{R}^{3\times 3}$ we define $\anti z=( \anti z_1\,|\, \anti z_2\,|\, \anti z_3)\in \mathbb{R}^{3\times 9}$. For a given matrix $M\in \mathbb{R}^{2\times 2}$ we define the lifted quantity $
M^\flat =\begin{footnotesize}\begin{pmatrix}
M_{11}& M_{12}&0 \\
M_{21}&M_{22}&0 \\
0&0&0
\end{pmatrix}\end{footnotesize}
\in \mathbb{R}^{3\times 3}.
$
	
	Let $\Omega$ be an open domain of $\mathbb{R}^{3}$. The usual Lebesgue spaces of square integrable functions, vector or tensor fields on $\Omega$ with values in $\mathbb{R}$, $\mathbb{R}^3$ or $\mathbb{R}^{3\times 3}$, respectively will be denoted by ${\rm L}^2(\Omega)$. Moreover, we introduce the standard Sobolev spaces 
	$
	{\rm H}^1(\Omega)\,=\,\{u\in {\rm L}^2(\Omega)\, |\, \D \, u\in {\rm L}^2(\Omega)\}, $ 
	${\rm H}({\rm curl};\Omega)\,=\,\{v\in {\rm L}^2(\Omega)\, |\, {\rm curl}\, v\in {\rm L}^2(\Omega)\}
	$
	of functions $u$ or vector fields $v$, respectively.  For vector fields $u=\left(    u_1, u_2,u_3\right)$ with  $u_i\in {\rm H}^{1}(\Omega)$, $i=1,2,3$,
	we define
	$
	\D  \,u:=\left(
	\D \,  u_1\,|\,
	\D \, u_2\,|\,
	\D \, u_3
	\right)^T,
	$
	while for tensor fields $P$ with rows in ${\rm H}({\rm curl}\,; \Omega)$, i.e.,
	$
	P=\begin{footnotesize}\begin{pmatrix}
	P^T.e_1\,|\,
	P^T.e_2\,|\,
	P^T. e_3
	\end{pmatrix}\end{footnotesize}^T$ with $(P^T.e_i)^T\in {\rm H}({\rm curl}\,; \Omega)$, $i=1,2,3$,
	we define
	$
	{\rm Curl}\,P:=\begin{footnotesize}\begin{pmatrix}
	{\rm curl}\, (P^T.e_1)^T\,|\,
	{\rm curl}\, (P^T.e_2)^T\,|\,
	{\rm curl}\, (P^T\, e_3)^T
	\end{pmatrix}\end{footnotesize}^T
	.
	$
	The corresponding Sobolev-spaces will be denoted by
	$
	{\rm H}^1(\Omega)$ and   ${\rm H}^1(\Curl;\Omega)$, respectively. We will use the notations: $\D  _\xi$, $\D  _x$, $\Curl_\xi$, $\Curl_x$ etc. to indicate the variables for which these quantities are calculated.

	\subsection{Geometrical nonlinear and physically linear Cosserat elastic 3D models}
	We consider an elastic material which  in its reference configuration fills the three dimensional domain $\Omega \subset R ^3$. 
	In the Cosserat theory, each point of the reference body is  endowed with three independent orthogonal directors, i.e., with a matrix field $\overline{R} :\Omega \rightarrow \text{SO(3)}$ called the \textit{microrotation} tensor. Let us remark that while the tensor $ {\rm polar}(\D  \varphi )\in \text{SO(3)}$ of the polar decomposition of $F := \D  \varphi ={\rm polar}(\D  \varphi )\sqrt{(\D  \varphi )^T\D  \varphi }$ is not independent of $\varphi $ \cite{nefffichle,borisonfischle,Neff_lankei}, the tensor $\overline{R} $ in the Cosserat theory is independent of $\D  \varphi $. In other words, in general, $\overline{R} \neq {\rm polar}(\D  \varphi )$.
	In  geometrical nonlinear and physically linear Cosserat elastic 3D models, the deformation $\varphi $ and the microrotation $\overline{R} $ are the solutions of  the following \textit{nonlinear minimization problems} on $\Omega $:
	\begin{align}\label{energy model}
	I(\varphi ,F , \overline{R} ,  \partial_{x_i}\overline{R} )=\dd\int_{\Omega } \Big[W_{\rm strain}(F,\overline{R} )+{W}_{\rm Cosserat- curv}(\overline{R},  \partial_{x_i}\overline{R} )\Big]\; dV   \quad \mapsto \min. \qquad\text{w.r.t}\quad (\varphi ,\overline{R} ),
	\end{align}
	where $	F=\D  \varphi$ represents the deformation gradient, $W_{\rm strain}(F,\overline{R} )$ is the strain energy, ${W}_{\rm Cosserat- curv}(\overline{R},  \partial_{x_i}\overline{R} )$ is the  Cosserat curvature (bending) energy and 	and $dV$ denotes the  volume element in the $\Omega $-configuration.  	For  simplicity of exposition we consider that the  {external loadings}  are not present and we have only Dirichlet-type boundary conditions for $\varphi$.

	  In this paper, the strain energy is considered to be  a general isotropic quadratic energy (physically linear) in terms of the non-symmetric Biot-type stretch tensor $\overline U  :\,=\,\dd\overline{R}^T  F \in \mathbb{R} ^{3\times 3}$ (the first Cosserat deformation tensor), i.e.,
	\begin{align}
	W_{\rm strain}(F,\overline{R} )=W_{\rm{mp}}(\overline U  ):\,=\,&\mu\,\lVert \dev\,\text{sym}(\overline U  -\id_3)\rVert^2+\mu_{\rm c}\,\lVert \text{skew}(\overline U  -\id_3)\rVert^2+
	\frac{\kappa}{2}\,[{\rm tr}(\text{sym}(\overline U  -\id_3))]^2\, ,
	\end{align}
	while the Cosserat curvature (bending) energy ${W}_{\rm Cosserat- curv}(\overline{R},  \partial_{x_i}\overline{R})$ is considered to be  isotropic  in terms of  $\overline{R}$ and quadratic in the following curvature strain candidates 
	\begin{align}\label{energy refeconf}
	\mathfrak{K} :=\,&\overline{R} ^T\D \overline{R} =\Big(\overline{R} ^T\partial_{x_1}\overline{R} ,\overline{R} ^T\partial_{x_1}\overline{R} ,\overline{R} ^T\partial_{x_1}\overline{R} \Big)\in \R^{3\times 9},\notag\\
	\widehat{\mathfrak{K}} :=\,&\big(\overline{R} ^T\D (\overline{R} .e_1),R^T\D (\overline{R} .e_2),R^T\D   (\overline{R} .e_3)\big)\in \R^{3\times 9},\\
	\alpha :\,=\,&\overline{R} ^T\, \Curl  \,\overline{R} \in \mathbb{R} ^{3\times 3}\qquad \qquad\qquad  \qquad\qquad\qquad\qquad\qquad  \small{\textrm{(the second order  dislocation density tensor \cite{Neff_curl08})}}\, ,\notag\\
	   \Gamma :=\,& \Big(\mathrm{axl}(\overline{R} ^T\,\partial_{x_1} \overline{R} )\,|\, \mathrm{axl}(\overline{R} ^T\,\partial_{x_2} \overline{R} )\,|\,\mathrm{axl}(\overline{R} ^T\,\partial_{x_3} \overline{R} )\,\Big)\in \mathbb{R}^{3\times 3} \ \small{\text{\rm (the wryness tensor)}}.\notag
	\end{align}
	 The second order Cosserat deformation tensor $\Gamma$ (the wryness tensor) was considered 	as a Lagrangian strain measure for curvature-orientation change \cite{Pietraszkiewicz04}) since the introduction of the Cosserat model \cite{Cosserat09}, the second order dislocation density tensor $\alpha$ is in a direct relation to the wryness tensor via Nye's formulas and energetically controls $\mathfrak{K}$ (see \cite{Neff_curl08,lankeitneff}), the tensor $\mathfrak{K}$ represents a first impulse choice while $\widehat{\mathfrak{K}}$ is an ad hoc choice which is not suitable for an isotropic quadratic curvature energy. Let us notice that all the mentioned curvature tensors are frame-indifferent by definition, i.e., they remain invariant after the change $\overline{R}\to  \overline{Q}\,\overline{R}$, with $\overline{Q}\in {\rm SO}(3)$ constant. In addition, $\Gamma$, $\alpha$ and $\mathfrak{K}$ are isotropic, property which are not shared with $\widehat{\mathfrak{K}} $.

	 The suitable form of the isotropic Cosserat-curvature energy is discussed in Section \ref{conections}. However, let us already announce  that from our point of view, the most suitable expression for analytical computations of the general isotropic energy quadratic in $\overline{R}$ is
	\begin{align}
{W}_{\rm Cosserat- curv}(\overline{R},  \partial_{x_i}\overline{R} )=
	W_{\rm{curv}}( \alpha ):\,=\,&\mu\,{L}_{\rm c}^2\left( b_1\,\lVert \text{sym}\, \alpha \rVert^2+b_2\,\lVert \text{skew}\, \alpha \rVert^2+  \frac{b_3}{4}\,
	[{\rm tr}(\alpha )]^2\right)\\
	=\,&\mu\,{L}_{\rm c}^2\left( b_1\,\lVert \text{sym}\, \Gamma \rVert^2+b_2\,\lVert \text{skew}\, \Gamma \rVert^2+  b_3\,
	[{\rm tr}(\Gamma )]^2\right). \notag
	\end{align}
	  The parameters $\mu\,$ and $\lambda$ are the elasticity \textit{Lam\'e-type\footnote{In the sense that if Cosserat effects are ignored $\mu\,$ and $\lambda$ are the Lam\'e coefficients of classical isotropic elasticity}}
constants, $\kappa=\displaystyle\frac{2\,\mu\,+3\,\lambda}{3}$ is the \textit{infinitesimal bulk modulus}, $\mu_{\rm c}> 0$ is the \textit{Cosserat couple modulus} and $L_c>0$ is the \textit{internal length} and responsible for \textit{size effects} in the sense that smaller samples are relatively stiffer than larger samples.  If not stated otherwise, we assume, here, that $\mu\,>0$, $\kappa>0$, $\mu_{\rm c}> 0$. The Cosserat couple modulus $\mu_c$ controls the deviation of the microrotation $\overline{R}$ from the continuum rotation $\text{polar}(\D   \varphi)$ in the polar decomposition of $\D   \varphi=\text{polar}(\D   \varphi)\cdot\sqrt{\D   \varphi^T \D  \varphi}$.
	For $\mu_c\to \infty$ the constraint $R=\text{polar}(\D \varphi)$ is generated and the model would turn into a Toupin couple stress model. We also assume that $b_1>0, b_2>0$ and $b_3>0$, which assures the \textit{coercivity}  and \textit{convexity of the curvature} energy \cite{neff2007geometrically1}.

	\subsection{More on Cosserat-curvature strain measures}\label{conections}

	\subsubsection{The curvature tensor $\mathfrak{K}=\overline{ R }^T {\rm D} \overline{ R } \in\mathbb{R}^{3\times9}$ }
	A first choice for a curvature strain tensor is  the third order elastic Cosserat curvature tensor  \cite{Cosserat08,Cosserat09,Cosserat1896,Cosserat09b}
	\begin{align}
{\mathfrak{K}}:=&\dd\overline{ R }^T {\rm D}\, \overline{ R }=\overline{ R }^T\,(\partial_{x_1} \overline{ R }\,|\, \partial_{x_2} \overline{ R }\,|\, \partial_{x_3} \overline{ R })=(\overline{ R }^T\,\partial_{x_1} \overline{ R }\,|\, \overline{ R }^T\,\partial_{x_2} \overline{ R }\,|\,\overline{ R }^T\, \partial_{x_3} \overline{ R }) \in\mathbb{R}^{3\times9},
	\end{align}
	and the curvature energy given by observing that 
$
	\dd \widetilde{W}_{\rm{curv}}({\mathfrak{K}}):=\dd a_1\,\|{\mathfrak{K}}\|^2\, . \notag
$
	This one-parameter choice is motivated \cite{gastelneff} by  
	$
	{\mathfrak{K}}\equiv\text{skew}\, {\mathfrak{K}},
$
	 since $\overline{ R }^T\partial_{x_i} \overline{ R }\in \mathfrak{so}(3)$, $i=1,2,3,$ and
	\begin{align}
	\text{sym}\, {\mathfrak{K}}&=\big(\text{sym}(\overline{ R }^T\partial_{x_1} \overline{ R })\,|\,\text{sym}(\overline{ R }^T\partial_{x_2} \overline{ R })\,|\,\text{sym}(\overline{ R }^T\partial_{x_3} \overline{ R }) \big)=(0_3\,|\,0_3\,|\,0_3),\notag\vspace{1.5mm}\\
	\tr {\mathfrak{K}}&=\tr(\overline{ R }^T\partial_{x_1} \overline{ R })+\tr(\overline{ R }^T\partial_{x_2} \overline{ R })+\tr(\overline{ R }^T\partial_{x_3} \overline{ R }) =0,
	\notag\vspace{1.5mm}\\
	\text{skew}\, {\mathfrak{K}}&=\big(\text{skew}(\overline{ R }^T\partial_{x_1} \overline{ R })\,|\,\text{skew}(\overline{ R }^T\partial_{x_2} \overline{ R })\,|\,\text{skew}(\overline{ R }^T\partial_{x_3} \overline{ R }) \big).
	\end{align}
However, this is not the most general form of an quadratic isotropic energy in  $\overline{R}$ as will be seen later.

	The third order tensor
$
	 {\mathfrak{K}}=\Big(\overline{ R }^T\partial_{x_1} \overline{ R }\,|\, \overline{ R }^T\partial_{x_2} \overline{ R }\,|\,\overline{ R }^T\partial_{x_3} \overline{ R }\,\Big) \in \mathbb{R} ^{3\times 9}
	$
is usually replaced	by  the  {wryness tensor} $\Gamma = \Big(\mathrm{axl}(\overline{ R }^T\partial_{x_1} \overline{ R })\,|\, \mathrm{axl}(\overline{ R }^T\,\partial_{x_2} \overline{ R })\,|\,\mathrm{axl}(\overline{ R }^T\partial_{x_3} \overline{ R })\,\Big)$, since we have the one-to-one relations  
\begin{align}\label{kG}
\mathfrak{K}=\anti \Gamma, \qquad \qquad \Gamma=\axl \mathfrak{K},
\end{align}
due to the fact that $\overline{ R }^T\partial_{x_1} \overline{ R }\in \mathfrak{so}(3),$ $i=1,2,3$,
which in indices read
\begin{align}\label{kGi}
\mathfrak{K}_{ijk}=\overline{R}_{li}\frac{\partial \overline{R}_{lj}}{\partial x_k}, \qquad \qquad \mathfrak{K}_{ijk}=-\epsilon_{ijl} \Gamma_{lk}, \qquad \qquad \Gamma_{ik}=\frac{1}{2}\,\sum_{r,l=1}^3
\epsilon_{ilr} {\mathfrak{K}}_{lrk}.
\end{align}

	For a detailed discussion on various strain measures of the non-linear micropolar continua we refer to  \cite{Pietraszkiewicz09}.
 	
\begin{proposition}\label{propiso}
	A general isotropic quadratic energy depending on $\overline{ R}^T{\rm D} \overline{ R}\in \mathbb{R}^{3\times 9}$ has the form
	\begin{align}\label{curveg}
	\widetilde{W}(\mathfrak{K})
	=\,&b_1\,\lVert \textrm{\rm sym} \,\axl \mathfrak{K}\rVert^2+b_2\,\lVert \textrm{\rm skew}\,\axl \mathfrak{K}\rVert^2+4\,b_3\,
	[{\rm tr}(\axl \mathfrak{K})]^2\notag\\
	=\,&\ \ \ \, b_1\,\lVert \textrm{\rm sym} \,\big(\axl (\mathfrak{K}.e_1)\,|\,\axl (\mathfrak{K}.e_2)\,|\,\axl (\mathfrak{K}.e_3)\big)\rVert^2\\&+b_2\,\lVert \textrm{\rm skew}\,\big(\axl (\mathfrak{K}.e_1)\,|\,\axl (\mathfrak{K}.e_2)\,|\,\axl (\mathfrak{K}.e_3)\big)\rVert^2\notag\\&+b_3\,
	[{\rm tr}(\big(\axl (\mathfrak{K}.e_1)\,|\,\axl (\mathfrak{K}.e_2)\,|\,\axl (\mathfrak{K}.e_3)\big))]^2.\notag
	\end{align}
\end{proposition}
\begin{proof} The proof is based on the result  from \cite{EremeyevPi} and on the identities \eqref{kG}. Indeed, a quadratic energy in $\mathfrak{K}$ is a quadratic energy in $\Gamma$. Due to the results presented in \cite{EremeyevPi}, a quadratic isotropic energy written in terms of  $\Gamma$ is given by 
\begin{align}\label{isoE}	W(\Gamma)=\,&b_1\,\lVert \textrm{\rm sym} \,\Gamma\rVert^2+b_2\,\lVert \textrm{\rm skew}\,\Gamma\rVert^2+b_3\,
	[{\rm tr}(\Gamma)]^2.
	\end{align}
Using \eqref{kG}, the proof is complete.
	\end{proof}

	We can express the uni-constant isotropic curvature term as a positive definite quadratic form in terms of $\Gamma$, i.e.,
	\begin{align}
	\|{\rm D} \overline{ R}\|^2_{\R^{3\times 3\times 3}}&=\|\overline{ R}^T{\rm D} \overline{ R}\,\|_{\R^{3\times 3\times 3}}^2=\|\mathfrak{K}\|_{\R^{3\times 3\times 3}}^2=\|\overline{ R}^T\partial_{x_1}  \overline{ R}\|_{\R^{3\times 3}}^2+\|\overline{ R}^T\partial_{x_2}\overline{ R}\|_{\R^{3\times 3}}^2+\|\overline{ R}^T\partial_{x_3}\overline{ R}\|_{\R^{3\times 3}}^2\notag\\
	&=2\,\|\axl(\overline{ R}^T\partial_{x_1}  \overline{ R})\|_{\R^{3}}^2+2\,\|\axl(\overline{ R}^T\partial_{x_2}  \overline{ R})\|_{\R^{3}}^22\,\|\axl(\overline{ R}^T\partial_{x_3}  \overline{ R})\|_{\R^{3}}^2=2\,\|\Gamma\|_{\R^{3\times 3}}^2.
	\end{align}
	Therefore, 	a general positive definite quadratic isotropic curvature energy \eqref{curveg} in $\Gamma$
	is a  positive definite  quadratic form in terms of $\|\overline{ R}^T{\rm D} \overline{ R}\|^2_{\R^{3\times 3\times 3}}$, and vice versa. 
Thus, working with  a quadratic isotropic positive definite energy in terms of $\overline{ R}^T{\rm D} \overline{ R}$ is equivalent with working with a quadratic positive definite isotropic energy in terms of $\Gamma$. Since the expression of a isotropic  curvature energy has a simpler form in terms of $\Gamma$, in order to keep the calculations as simple as possible, we prefer to work with $\Gamma$. 

	\subsubsection{The curvature tensor $\alpha=\overline{R}^T\Curl \,\overline{R}\in \mathbb{R}^{3\times 3}$}\label{cta}

Another choice as curvature strain  is  the  \textit{second order dislocation density tensor} $\alpha$.
Comparing with $\overline{R}^T\D \,\overline{R}$, it simplifies considerably
the representation by allowing to use the orthogonal decomposition
\begin{align}\overline{R}^T\, \Curl \,\overline{R}=\alpha=\dev\, \sym \,\alpha+\skw\, \alpha+ \frac{1}{3}\,\tr(\alpha)\id_3.\end{align}

Moreover, it yields an equivalent control of spatial derivatives of rotations \cite{Neff_curl08} and allows us to write the curvature energy   in a fictitious Cartesian configuration in terms of the {wryness tensor}  \cite{Neff_curl08,Pietraszkiewicz04}
$
\Gamma\in  \mathbb{R} ^{3\times 3},
$
since (see \cite{Neff_curl08})  the following close relationship between the \textit{wryness tensor}
and the \textit{dislocation density tensor} holds
\begin{align}\label{Nye1}
\alpha =-\Gamma  ^T+\tr(\Gamma  )\, \id_3, \qquad\textrm{or equivalently,}\qquad \Gamma  =-\alpha ^T+\frac{1}{2}\tr(\alpha)\, \id_3.
\end{align}
Hence,  
\begin{align}
\sym \,\Gamma  =&-\sym\,\alpha+ \frac{1}{2}\tr(\alpha )\, \id_3, \qquad 
{\rm dev}\,\sym \,\Gamma  =-{\rm dev}\,\sym\,\alpha,\vspace{1.5mm}\notag\\
\skw \,\Gamma  =&-\skw(\alpha^T )=\skw\,\alpha ,\qquad\quad \ \ 
\tr(\Gamma  )=-\tr(\alpha )+ \frac{3}{2}\,\tr(\alpha )=\frac{1}{2}\tr(\alpha )
\end{align}
and
\begin{align}
\sym \,\alpha\,=\,&-\sym\,\Gamma+ \tr(\Gamma)\, \id_3, \quad
{\rm dev}\,\sym \,\alpha\,=\,-{\rm dev}\,\sym\, \Gamma,\quad
{\skw}
\,\alpha\,=\,\skw\,\Gamma,\quad 
\tr(\alpha)\,=\,2\,\tr(\Gamma).
\end{align}
In addition, from \cite{EremeyevPi} we have 
\begin{proposition}
	A general quadratic isotropic energy depending on $\alpha$ has the form
	\begin{align}\label{giso}
	W_{\rm{curv}}( \alpha )
	=\,&b_1\,\lVert \textrm{\rm sym}\, \alpha\rVert^2+b_2\,\lVert \textrm{\rm skew}\,\alpha\rVert^2+\frac{b_3}{4}\,
	[{\rm tr}(\alpha)]^2.
	\end{align}
\end{proposition}
\begin{proof} We use again that  a quadratic energy in $\alpha$ is a quadratic energy in $\Gamma$, i.e., due to \cite{EremeyevPi},  is given by 
	\eqref{isoE}. 
The proof is complete after using the  Nye's formulas \eqref{Nye1}.
\end{proof}
Since   a quadratic isotropic positive definite energy in terms of $\overline{ R}^T{\rm D} \overline{ R}$ is equivalent with  a quadratic positive definite isotropic energy in terms of $\Gamma$, considering $\alpha$ is equivalent with considering ${\mathfrak{K}}$, as long as a quadratic isotropic energy is used.  As we will see in the present paper, a quadratic curvature energy in terms of $\alpha$ is suitable for explicit calculations of homogenized curvature energy for shell models via the $\Gamma$-convergence method.

	\subsubsection{	The curvature (in fact: bending) tensor $\widehat{\mathfrak{K}}=\big(\overline{ R }^T \D  (\overline{ R }.e_1)\,|\,\overline{ R }^T \D  (\overline{ R }.e_2)\,|\,\overline{ R }^T \D  (\overline{ R }.e_3)\big)\in \mathbb{R} ^{3\times 9}$}
	
	The curvature tensor $\widehat{\mathfrak{K}}=\big(\overline{ R }^T \D  (\overline{ R }.e_1)\,|\,\overline{ R }^T \D  (\overline{ R }.e_2)\,|\,\overline{ R }^T \D  (\overline{ R }.e_3)\big)\in \mathbb{R} ^{3\times 9}$ is motivated from the flat Cosserat shell model
	\cite{neff2004geometricallyhabil,neff2004geometrically}. Indeed, in this setting, the general bending energy term appearing from an engineering ansatz through the thickness of the shell appears as
	\begin{align}\label{isoinc}
	\frac{h^3}{12}\left(	\mu\,\lVert \, \mathrm{sym} (\overline{R}^T\D (\overline{ R }.e_3))\rVert^2 +  \mu_{\rm c}\,\lVert \mathrm{skew} (\overline{R}^T\D (\overline{ R }.e_3))\rVert^2 +\,\dfrac{\lambda\,\mu}{\lambda+2\mu}\,\big[ \mathrm{tr}   (\overline{R}^T\D (\overline{ R }.e_3))\big]^2\right).
	\end{align}

	Motivated by this, in earlier papers \cite{neff2004geometrically,neff2004geometricallyhabil,neff2007geometrically,neff2007geometrically1}, as generalization of $\overline{R}^T\D(\overline{ R }.e_3)$, 	the  third order elastic Cosserat curvature tensor is considered in the form
		\begin{align}
		\widehat{\mathfrak{K}}=\big(\overline{ R }^T \D  (\overline{ R }.e_1)\,|\,\overline{ R }^T \D  (\overline{ R }.e_2)\,|\,\overline{ R }^T \D  (\overline{ R }.e_3)\big)=\overline{ R }^T \big(\D  (\overline{ R }.e_1), \D  (\overline{ R }.e_2), \D  (\overline{ R }.e_3)\big) \in \mathbb{R} ^{3\times 9} ,
		\end{align}
	treating the three directions $e_1, e_2, e_3$ equally,	and the curvature energy is taken to be
		\begin{align}
	\widehat{	W}_{\rm{curv}}( \widehat{\mathfrak{K}})=a_1\|\sym \,\widehat{\mathfrak{K}}\|^2+a_2\|\skew \widehat{\mathfrak{K}}\|^2+a_3[\tr (\widehat{\mathfrak{K}})]^2.
		\end{align}
We mean by $\,\widehat{\mathfrak{K}}.e_i=\overline{ R }^T \D  (\overline{ R }.e_i)\,$ and
\begin{align}\|\text{sym}\,\widehat{\mathfrak{K}}\|^2&=\displaystyle \sum_{i=1}^3\|\text{sym}\,\widehat{\mathfrak{K}}.e_i\|^2= \sum_{i=1}^3\|\text{sym}\,(\overline{ R }^T \D  (\overline{ R }.e_i))\|^2, \notag\\ \|\text{skew}\,\widehat{\mathfrak{K}}\|^2&=\displaystyle \sum_{i=1}^3\|\text{skew}\,\widehat{\mathfrak{K}}.e_i\|^2= \sum_{i=1}^3\|\text{skew}\,(\overline{ R }^T \D  (\overline{ R }.e_i))\|^2,\\
[\tr(\widehat{\mathfrak{K}})]^2&=\displaystyle \sum_{i=1}^3[\tr(\widehat{\mathfrak{K}}.e_i)]^2= \sum_{i=1}^3[\tr(\overline{ R }^T \D  (\overline{ R }.e_i))]^2\notag.\end{align}
		However, this curvature energy has now three abstract orthogonal preferred directions, which makes it only cubic and not isotropic, as we will see.
		
		There does not exist an analysis to show that this is the most general form of an isotropic energy depending on ${\widehat{\mathfrak{K}}}$.
		 Actually, as we will see in the following, for general positive  values for $(\alpha_i^1, \alpha_i^2,\alpha_i^3)$ the energies of the form \eqref{isoinc} are anisotropic.  We simplify the discussion  by considering only the energy $\lVert  \overline{R}^T\D(\overline{R}.e_i)\rVert^2$. After the transformation $\overline{R}\to \overline{R}\,\overline{Q}$, with $\overline{Q}=(e_1\,|\,e_3\,|\, e_2)\in {\rm SO}(3)$ constant, we have 
		\begin{align}
		\lVert  \overline{Q}^T\overline{R}^T\D(\overline{R}\,\overline{Q}.e_3)\rVert^2=	\lVert  \overline{R}^T\D(\overline{R}\,\overline{Q}.e_3)\rVert^2=\lVert  \overline{R}^T\D(\overline{R}.e_2)\rVert^2\neq \lVert  \overline{R}^T\D(\overline{R}.e_3)\rVert^2.
		\end{align}

		As regards the direct relation between $\|\widehat{\mathfrak{K}}\|^2=\sum_{i=1}^3\big\|\overline{ R}^T\D (\overline{ R}.e_i)\big\|^2$ and $\|\alpha\|^2$ we have
	\begin{align}
	\sum_{i=1}^3\big\|\overline{ R}^T\D (\overline{ R}.e_i)\big\|^2_{\R^{3\times 3}}&=\sum_{i=1}^3\big\|\D (\overline{ R}.e_i)\big\|^2_{\R^{3\times 3}}=\big\|\D \overline{ R}\big\|^2_{\R^{3\times 3}}\notag\\&=1\cdot\|\dev\,\sym \overline{ R}^T\Curl \overline{ R}\|^2_{\R^{3\times 3}}+1\cdot\|\skew \overline{ R}^T\Curl \overline{ R}\|^2_{\R^{3\times 3}}+\frac{1}{12}\cdot[\tr(\overline{ R}^T\Curl \overline{ R})]^2\\
	&\geq c_+ \big\|\Curl \overline{ R} \,\big\|^2_{\R^{3\times 3}},\notag
	\end{align}
	where $c_+>0$ is a constant.
		Since 	a coercive   curvature energy in $\widehat{\mathfrak{K}}$  is completely   controlled by $\sum_{i=1}^3\big\|\overline{ R}^T\D (\overline{ R}.e_i)\big\|^2_{\R^{3\times 3}},
	$ a  positive definite quadratic isotropic energy  of the form in terms of $\alpha=\overline{ R}^T\Curl \overline{ R}$  (equivalently on the wryness tensor $\Gamma$)
	is a  positive definite  quadratic form in terms of $\|\overline{ R}^T{\rm D} \overline{ R}\|^2_{\R^{3\times 3\times 3}}$, and vice versa. 
	Hence, a quadratic positive definite energy in terms of $\widehat{\mathfrak{K}}$ is energetically equivalent with a quadratic positive definite energy in terms of $\alpha$ (and $\Gamma$).

Let us remark that 	both
 $\Big(\partial_{x_1}\overline{R} \,|\,\partial_{x_2}\overline{R} \,|\,\partial_{x_3}\overline{R} \Big)\in \R^{3\times 9},$ $\big(\D (\overline{R} .e_1),\D (\overline{R} .e_2),\D   (\overline{R} .e_3)\big)$ contain the same terms, $\frac{\partial \overline{R}_{ij}}{\partial x_k}$, $i,j,k=1,2,3$, but differently ordered. By multiplication with $\overline{R}^T$ of $\Big(\partial_{x_1}\overline{R} \,|\,\partial_{x_2}\overline{R} \,|\,\partial_{x_3}\overline{R} \Big)\in \R^{3\times 9}$ and $\big(\D (\overline{R} .e_1)\,|\,\D (\overline{R} .e_2)\,|\,\D   (\overline{R} .e_3)\big)$ we obtain $\mathfrak{K}$ and $\widehat{\mathfrak{K}}$, respectively, i.e. 
	$
	\mathfrak{K}_{ijk}=\overline{R}_{li}\frac{\partial \overline{R}_{lj}}{\partial x_k}, \ 	\widehat{\mathfrak{K}}_{ijk}=\overline{R}_{li}\frac{\partial \overline{R}_{lk}}{\partial x_j}, i,j,k=1,2,3
	$
	and
	$
	\mathfrak{K}_{ijk}=	\widehat{\mathfrak{K}}_{ikj},\  i,j,k=1,2,3.
	$ 
	 We have the following relation between $\widehat{\mathfrak{K}}$ and $\Gamma$
	 \begin{align}\label{kGiG}
	 \widehat{\mathfrak{K}}_{ijk}=\mathfrak{K}_{ikj}=-\epsilon_{ikl} \Gamma_{lj}, \qquad \qquad \Gamma_{ik}=\frac{1}{2}\,\sum_{r,l=1}^3
	 \epsilon_{ilr} {\mathfrak{K}}_{lrk}=\frac{1}{2}\,\sum_{r,l=1}^3
	 \epsilon_{ilr} {\widehat{\mathfrak{K}}}_{lkr}.
	 \end{align}

	 Let us introduce the operator $\mathcal{A}:\mathbb{R}^{3\times 9}\to \mathbb{R}^{3\times 9}$ by
	 $
	 (\mathcal{A}.\widehat{\mathfrak{K}})_{ijk}=\widehat{\mathfrak{K}}_{ikj}.
	 $

	\begin{proposition}
		A general isotropic energy depending on $\widehat{\mathfrak{K}}=\big(\overline{ R }^T \D  (\overline{ R }.e_1)\,|\,\overline{ R }^T \D  (\overline{ R }.e_2)\,|\,\overline{ R }^T \D  (\overline{ R }.e_3)\big)\in \mathbb{R}^{3\times 9}$ has the form
	\begin{align}\label{curvegkt}
	\widetilde{W}(\widehat{\mathfrak{K}})
	=\,&b_1\,\lVert \textrm{\rm sym} \,\axl(\mathcal{A}.\widehat{\mathfrak{K}})\rVert^2+b_2\,\lVert \textrm{\rm skew}\,\axl(\mathcal{A}.\widehat{\mathfrak{K}})\rVert^2+b_3\,
	[{\rm tr}(\axl(\mathcal{A}.\widehat{\mathfrak{K}}))]^2.
	\end{align}
	\end{proposition}

\begin{proof} Using \eqref{propiso}, we have that  a quadratic isotropic energy in ${\widehat{\mathfrak{K}}}$ is given by 
	\begin{align}\label{curvegk}
\widetilde{W}(\widehat{\mathfrak{K}})
=\,&b_1\,\lVert \textrm{\rm sym} \,\axl \mathfrak{K}\rVert^2+b_2\,\lVert \textrm{\rm skew}\,\axl \mathfrak{K}\rVert^2+b_3\,
[{\rm tr}(\axl \mathfrak{K})]^2.
\end{align}
Since  $
\mathcal{A}.\widehat{\mathfrak{K}}=\mathfrak{K}$, the proof is complete.
\end{proof}

Let us remark that comparing to \eqref{isoinc}, a general isotropic energy depending on $\widehat{\mathfrak{K}}$, i.e., \eqref{curvegkt} is different since we have the summation of different products between the elements of $\overline{R}^T$ and $\D \overline{R}$, due to the action of the axial operator together with the operator $\mathcal{A}$. 

From \eqref{isoinc} one could  obtain an isotropic energy by setting 
	\begin{align}\label{isoincav}
\int_{\widetilde{Q}\in {\rm SO}(3)}\frac{h^3}{12}\sum_{i=1}^3\left(	\mu\,\lVert \, \mathrm{sym} (\overline{R}^T\D (\widetilde{Q}.e_i))\rVert^2 +  \mu_{\rm c}\,\lVert \mathrm{skew} (\overline{R}^T\D (\widetilde{Q}.e_i))\rVert^2 +\,\dfrac{\lambda\,\mu}{\lambda+2\mu}\,\big[ \mathrm{tr}   (\overline{R}^T\D (\widetilde{Q}.e_i))\big]^2\right),
\end{align}i.e., averaging over all directions.

\section{Homogenized curvature energy for the flat Cosserat-shell model via $\Gamma$-convergence}\label{flat}\setcounter{equation}{0}

 Let us consider an elastic material which  in its reference configuration fills the three dimensional \textit{flat shell-like thin} domain $\Omega_h=\omega\times\big[-\frac{h}{2},\frac{h}{2}\big]$, and $\omega\subset \mathbb{R}^2$  a bounded domain with Lipschitz boundary $\partial\omega$. The scalar $0<h\ll1$ is called \textit{thickness} of the shell.
  
 Due to the discussion from  Subsection \ref{conections}, in this paper we consider the Cosserat-curvature energy in terms of the wryness tensor $\Gamma$ in the form
 \begin{align}\label{curvenerg}
 \widetilde{W}_{\rm curv}(\Gamma)&\,=\,\mu\,{L}_{\rm c}^2 \left( b_1\,\lVert\sym \,\Gamma\rVert^2+b_2\,\lVert \skew \,\Gamma\rVert^2+\,b_3\,
 [{\rm tr}(\Gamma)]^2\right)\,.
 \end{align}

  In order to apply  the methods of $\Gamma$-convergence for constructing the variational problem on $\omega$ of the flat Cosserat-shell model,  the first step is to transform our problem further from $\Omega_h$ to a \textit{domain} with fixed thickness  $\Omega_1=\omega\times[-\frac{1}{2},\frac{1}{2}]\subset\mathbb{R}^3,\;\omega\subset \mathbb{R}^2$. For this goal, scaling of the variables (dependent/independent) would be the first step. 
   In all our computations the mark $\cdot^\natural$ indicates the nonlinear scaling and the mark $\cdot_h$ indicates that the assigned quantity depends on the thickness $h$.
  In a first step we will apply the nonlinear scaling to the deformation. For $\Omega_1=\omega\times\Big[-\displaystyle\frac{1}{2},\frac{1}{2}\Big]\subset\mathbb{R}^3$, $\omega\subset \mathbb{R}^2$, we define the scaling transformations
 \begin{align}\label{nonlinear sca}
 \nonumber\zeta\col&\; \eta\in\Omega_1\mapsto \mathbb{R}^3\,, \qquad \zeta(\eta_1,\eta_2,\eta_3):=(\eta_1,\eta_2,h\,\eta_3)\,,\quad 
 \zeta^{-1}\col\;x\in\Omega_h\mapsto\mathbb{R}^3\,, \qquad\zeta^{-1}(x_1,x_2,x_3):=(x_1,x_2,\frac{x_3}{h})\,,
 \end{align}
 with $\zeta(\Omega_1)=\Omega_h$. By using the  above transformations (\ref{nonlinear sca}) we obtain the formula for the transformed deformation $\varphi$ as 
 \begin{align}
 \nonumber\varphi(x_1,x_2,x_3)&=\varphi^\natural(\zeta^{-1}(x_1,x_2,x_3))\, \quad \forall x\in \Omega_h \,; \qquad \varphi^\natural(\eta)=\varphi(\zeta(\eta))\quad \forall \eta\in \Omega_1\,,\\
 \D _x\varphi(x_1,x_2,x_3)&=\begin{pmatrix}\vspace{0.2cm}
 \partial_{\eta_1}\varphi_{1}^\natural(\eta)& \partial_{\eta_2}\varphi_{1}^\natural(\eta)&\displaystyle\frac{1}{h}\partial_{\eta_3}\varphi_{1}^\natural(\eta)\\\vspace{0.2cm}
 \partial_{\eta_1}\varphi_{2}^\natural(\eta)& \partial_{\eta_2}\varphi_{2}^\natural(\eta)&\displaystyle\frac{1}{h}\partial_{\eta_3}\varphi_{2}^\natural(\eta)\\\vspace{0.2cm}
 \partial_{\eta_1}\varphi_{3}^\natural(\eta)& \partial_{\eta_2}\varphi_{3}^\natural(\eta)&\displaystyle\frac{1}{h}\partial_{\eta_3}\varphi_{3}^\natural(\eta)\\
 \end{pmatrix}=\D _\eta^h\varphi^\natural(\eta)=:F^\natural_h \,.
 \end{align}
 Now we will do the same process for the microrotation tensor $  \overline{R}_h^\natural\col\Omega_1\rightarrow \text{SO(3)}$ 
 \begin{align}
  \overline{R}(x_1,x_2,x_3)&=  \overline{R}_h^\natural(\zeta^{-1}(x_1,x_2,x_3))\, \qquad \forall x\in \Omega_h\,;\; \;  \overline{R}_h^\natural(\eta)= \overline{R}(\zeta(\eta))\,,\; \quad\forall \eta\in \Omega_1\,.
 \end{align}

 With this, the non-symmetric stretch tensor expressed in a point of $\Omega_1$ is given by
 \begin{align}
 \UEN =  \overline{R}_h^{\natural,T}F^\natural_h =  \overline{R}_h^{\natural,T}\D _\eta^h\varphi^\natural(\eta)\,.
 \end{align}
 and
 \begin{align}\label{gs}
 \Gamma_{e,h}^\natural& =\Big(\mathrm{axl}( \overline{R}_h^{\natural,T}\,\partial_{\eta_1}  \overline{R}_h^\natural)\,|\, \mathrm{axl}( \overline{R}_h^{\natural,T}\,\partial_{\eta_2}  \overline{R}_h^\natural)\,|\,\frac{1}{h}\mathrm{axl}( \overline{R}_h^{\natural,T}\,\partial_{\eta_3}  \overline{R}_h^\natural)\,\Big)
 .
 \end{align}

 
 The next step, in order to apply the $\Gamma$-convergence technique, is to transform the minimization problem onto the \textit{fixed domain} $\Omega_1$, which is independent of the thickness $h$.
 According to the results from the previous subsection, we have found that the original three-dimensional variational problem \eqref{energy refeconf} is equivalent to the following minimization problem on $\Omega_1$
 \begin{align}\label{energy model.fictituos.ond omega1}
 I^\natural_h(\varphi^\natural,\D _\eta^h\varphi^\natural,  \overline{R}_h^\natural,\Gamma_{e,h}^\natural)
 & =\int_{\Omega_1}\; h \;\Big[\Big({{W}}_{\rm mp}(U_{e,h}^\natural)+{\widetilde{W}}_{\rm curv}(\Gamma_{e,h}^\natural)\Big)\Big]\;dV_\eta\quad 
 \mapsto \quad \min\;\text{w.r.t}\; (\varphi^\natural,  \overline{R}_h^\natural)\,,
 \end{align}
 where 
 \begin{align}\label{boun co}
 \nonumber {{W}}_{\rm mp}(U_{e,h}^\natural)&=\;\mu\,\norm{\sym (U_{e,h}^\natural-\id_3)}^2+\mu_c\,\norm{\skew (U_{e,h}^\natural-\id_3)}^2 + \frac{\lambda}{2}[\tr(\sym(U_{e,h}^\natural-\id_3))]^2\,,\\
 \widetilde{W}_{\rm curv}(\Gamma_{e,h}^\natural)&\,=\,\mu\, {L}_{\rm c}^2 \left( a_1\,\lVert \dev \,\text{sym} \,\Gamma_{e,h}^\natural\rVert^2+a_2\,\lVert \text{skew} \,\Gamma_{e,h}^\natural\rVert^2+\,a_3\,
 [{\rm tr}(\Gamma_{e,h}^\natural)]^2\right)\\
&\,=\,\mu\,{L}_{\rm c}^2 \left( b_1\,\lVert\sym \,\Gamma_{e,h}^\natural\rVert^2+b_2\,\lVert \skew \,\Gamma_{e,h}^\natural\rVert^2+\,b_3\,
 [{\rm tr}(\Gamma_{e,h}^\natural)]^2\right)\,,\notag
 \end{align}
 where  $a_1=b_1,$ $a_2=b_2$ and $a_3=\frac{12b_3-b_1}{3}$.
  
  In the article \cite{neff2007geometrically1} one aim of the authors was to  to find the $\Gamma$-limit of the family of functional which is related to 
  \begin{align}\label{energyinnewarea 1}
  \mathcal{I}_h^\natural(\varphi^\natural,  \D ^h_\eta \varphi^\natural, \overline{R}_h^\natural,\Gamma_h^\natural) &=
  \begin{cases}
  \displaystyle\frac{1}{h}\,I_h^\natural(\varphi^\natural,  \D ^h_\eta \varphi^\natural,  \overline{R}_h^\natural,\Gamma_h^\natural)\quad &\text{if}\;\; (\varphi^\natural, \overline{R}_h^\natural)\in \mathcal{S}',
  \\
  +\infty\qquad &\text{else in}\; X,
  \end{cases}
  \end{align}
  where 
   \begin{align}\label{sets}
  X&:=\{(\varphi^\natural, \overline{R}_h^\natural)\in {\rm L}^2(\Omega_1,\mathbb{R}^3)\times {\rm L}^{2}(\Omega_1, \text{SO(3)})\}\,,\\
  \nonumber\mathcal{S}'&:=\{(\varphi, \overline{R}_h)\in {\rm H}^{1}(\Omega_1,\mathbb{R}^3)\times { {{\rm H}^{1}}}(\Omega_1,\text{SO(3)})\,\big|\,\, \varphi|_{\partial \Omega_1}(\eta)=\varphi^\natural_d(\eta)\}\,.
  \end{align}
 That means, to obtain an energy functional expressed only in terms of the weak limit of a  subsequence of  $(\varphi_{h_j}^\natural, \overline{R}_{h_j}^\natural)\in X$, when $h_j$ goes to zero. In other words, as we will see, to construct an energy function depending only on quantities definite on the planar midsurface $\omega$.
 However, in \cite{neff2007geometrically1}  the authors have considered a different Cosserat-curvature energy based on the Cosserat-curvature tensor
  $\widehat{\mathfrak{K}}=(\overline{R}^T\D  (\overline{R}.e_1) R_1,\overline{R}^T\D   (\overline{R}.e_2),\overline{R}^T\D   (\overline{R}.e_3))\in \R^{3\times 3\times3}$, which in the simplest form  reads
 \begin{align}\label{nch}
 \widehat{W}_{\text{curv}}(\widehat{\mathfrak{K}})=\mu\frac{L_c^2}{12}\Big(\alpha_1\norm{\sym \widehat{\mathfrak{K}}}^2+\alpha_2\norm{\skew\widehat{\mathfrak{K}}}^2+\alpha_3\tr(\widehat{\mathfrak{K}})^2\Big)\,,
 \end{align}
and no explicit form of the homogenized curvature energy has been computed (perhaps even not possible to be computed for curved initial configuration). In fact $\widehat{\mathfrak{K}}$ is not isotropic and has to be avoided in an isotropic model, as seen above.

 In order to construct  the $\Gamma$-limit there is the need to solve two auxiliary optimization problems, i.e.,
 \begin{itemize}
 	\item[O1:] the  optimization problem which for each pair $(m, \overline{R}_{0})$, where $m:\omega\to \mathbb{R}^3$,  $\overline{R}_{0}:\omega\to {\rm SO}(3)$ defines the homogenized membrane energy
 \begin{align}
 {W}_{\rm mp}^{{\rm hom, plate}}
 (\mathcal{E}^{\rm plate}_{m,\overline{R}_{0}}):=\inf_{\widetilde{d}\in \mathbb{R}^3} {W}_{\rm mp}\Big(\overline{R}^{T}_{0}(\D  m|\widetilde{d})\Big)=\inf_{\widetilde{d}\in \mathbb{R}^3} {W}_{\rm mp}\Big(\mathcal{E}^{\rm plate}_{m,\overline{R}_{0}}-(0|0|\widetilde{d})\Big).
 \end{align}
 where 
 $
 \mathcal{E}^{\rm plate}_{m,\overline{R}_{0}} :=\overline{R}^{T}_{0}(\D  m |0)-\id_2^\flat\,
$
denotes the \textit{elastic strain tensor} for the flat Cosserat-shell model.
	\item[O2:] the  optimization problem which for each   $\overline{R}_{0}:\omega\to {\rm SO}(3)$ defines the   homogenized curvature energy 
	\begin{align} \widetilde{W}^{{\rm hom, plate}}_{\rm curv}(\mathcal{K}_{\overline{R}_{0}}^{\rm plate} ):&=\widetilde{W}_{\rm curv}^*\Big(\mathrm{axl}(\overline{R}_{0}^T\,\partial_{\eta_1} \overline{R}_{0})\,|\, \mathrm{axl}(\overline{R}_{0}^T\,\partial_{\eta_2} \overline{R}_{0})\,|\,\axl \,(A^*)\,\Big)\\
	&=\inf_{A\in \mathfrak{so}(3)}\widetilde{W}_{\rm curv}\Big(\mathrm{axl}(\overline{R}_{0}^T\,\partial_{\eta_1} \overline{R}_{0})\,|\, \mathrm{axl}(\overline{R}_{0}^T\,\partial_{\eta_2} \overline{R}_{0})\,|\,\axl \,(A)\,\Big)\notag
	\end{align}
	where 
	$
	\mathcal{K}_{\overline{R}_{0}}^{\rm plate}  :\,=\,  \Big(\mathrm{axl}(\overline{R}_{0}^T\,\partial_{x_1} \overline{R}_{0})\,|\, \mathrm{axl}(\overline{R}_{0}^T\,\partial_{x_2} \overline{R}_{0})\,|0\Big)\not\in {\rm Sym}(3) $
	denotes	the elastic  bending-curvature tensor for the flat Cosserat-shell model.
	\end{itemize}

The first optimisation problem O1 was solved in \cite{neff2007geometrically1} giving
   \begin{align}\label{hominf_2D}
  {W}_{\rm mp}^{{\rm hom, plate}}
  (\mathcal{E}^{\rm plate}_{m,\overline{R}_{0}})
  &=W_{\mathrm{shell}}\big(   [\mathcal{E}_{m,\overline{Q}_{e,0}}^{\rm plate}]^{\parallel} \big)+\frac{2\,\mu\, \,  \mu_{\rm c}}{\mu_c\,+\mu\,}\lVert [\mathcal{E}_{m,\overline{Q}_{e,0}}^{\rm plate}]^{\perp}\rVert^2,\notag
  \end{align}
 with the orthogonal decomposition in the tangential plane and in the normal direction\footnote{Here, for vectors $\xi,\eta\in\mathbb{R}^n$, we have considered the tensor product
 	$(\xi\otimes\eta)_{ij}=\xi_i\,\eta_j$. Let us denote by $\overline{R}_i$ the columns of the matrix $\overline{R}$, i.e., $\overline{R}=\big(\overline{R}_1\,|\,\overline{R}_2\,|\,\overline{R}_3\big)$, $\overline{R}_i=\overline{R}\, e_i$.
 	Since $(\id_3-e_3\otimes e_3)\overline{R}^{T}=\big(\overline{R}_1\,|\,\overline{R}_2\,|\,0\big)^T$, it follows that $ [\mathcal{E}_{m,\overline{Q}_{e,0}}^{\rm plate}]^\parallel=\big(\overline{R}_1\,|\,\overline{R}_2\,|\,0\big)^T(\D  m|0)-\id_2^\flat=\big(\big(\overline{R}_1\,|\,\overline{R}_2\big)^T\,\D  m\big)^\flat-\id_2^\flat$, while
 	\begin{align}
 	[\mathcal{E}_{m,\overline{Q}_{e,0}}^{\rm plate}]^\perp=(0\,|\,0\,|\,\overline{R}_3)^T(\D  m|0)=\begin{footnotesize}
 	\begin{pmatrix}
 	0&0&0\\
 	0&0&0\\
 	\langle \overline{R}_3,\partial_{x_1}m\rangle&\langle \overline{R}_3,\partial_{x_2}m\rangle&0
 	\end{pmatrix}\,.
 	\end{footnotesize}
 	\end{align} }
 \begin{align}
 [\mathcal{E}_{m,\overline{R}_{0}}^{\rm plate}]^\parallel\coloneqq(\id_3-e_3\otimes e_3) \,[\mathcal{E}_{m,\overline{R}_{0}}^{\rm plate}], \qquad \qquad [\mathcal{E}_{m,\overline{R}_{0}}^{\rm plate}]^\perp\coloneqq(e_3\otimes e_3)\,[\mathcal{E}_{m,\overline{R}_{0}}^{\rm plate}]\,
 \end{align}
  and
 \begin{align}
 W_{\mathrm{shell}}\big(   [\mathcal{E}_{m,\overline{R}_{0}}^{\rm plate}]^{\parallel} \big)=
 \, \mu\,\lVert  \mathrm{sym}\,   \,[\mathcal{E}_{m,\overline{R}_{0}}^{\rm plate}]^{\parallel}\rVert^2 +  \mu_{\rm c}\,\lVert \mathrm{skew}\,   \,[\mathcal{E}_{m,}^{\rm plate}]^{\parallel}\rVert^2 +\,\dfrac{\lambda\,\mu\,}{\lambda+2\,\mu\,}\,\big[ \mathrm{tr}    ([\mathcal{E}_{m,\overline{R}_{0}}^{\rm plate}]^{\parallel})\big]^2.
 \end{align}

  As regards the second optimization problem O2, in \cite{neff2007geometrically1} the authors had to solve a similar problem but corresponding to a curvature energy given by \eqref{nch}, i.e., the dimensionally reduced homogenized curvature energy is defined through the
  \begin{align}
  W_{\text{curv}}^{\text{hom, plate}}(\mathcal{A})=\inf_{u,v,w\in \mathbb{R}^3}\widehat{W}_{\text{curv}}\Big((\mathcal{A}e_1|u),(\mathcal{A}e_2|v),(\mathcal{A}e_3|w)\Big)\,,
  \end{align}
  where $\mathcal{A}:=(\overline{R}_{0}^T(\partial_{x_1}(\overline{R}_{0}e_1)|\partial_{x_2}(\overline{R}_{0}e_1)),\overline{R}_{0}^T(\partial_{x_1}(\overline{R}_{0})e_2|\partial_{x_2}(\overline{R}_{0}e_2)),\overline{R}_{0}^T(\partial_{x_1}(\overline{R}_{0}e_3)|\partial_{x_2}(\overline{R}_{0}e_3)))$. In this representation, calculating the homogenized energy looks more difficult and it was not explicitly done.

  In this section we show that  considering the curvature energy depending on the Cosserat-curvature tensor $\alpha$  (equivalently on the three-dimensional wryness tensor $\Gamma$) the calculation of the homogenized curvature energy (i.e., the solution of O2) is  easier and analytically achievable. 
   \begin{theorem}
   	The homogenized curvature energy for a flat Cosserat-shell model  is given by
   	\begin{align}\label{homflat}
   	W_{\text{\rm curv}}^{\text{\rm hom,plate}}(\Gamma  )
   	&=\mu L_c^2\Big(b_1\norm{\sym[\mathcal{K}_{\overline{R}_{0}}^{\rm plate}]^\parallel}^2+b_2\norm{\skew [\mathcal{K}_{\overline{R}_{0}}^{\rm plate}]^\parallel}^2+\frac{b_1b_3}{(b_1+b_3)}\tr([\mathcal{K}_{\overline{R}_{0}}^{\rm plate}]^\parallel)^2+\frac{2\,b_1b_2}{b_1+b_2}\norm{[\mathcal{K}_{\overline{R}_{0}}^{\rm plate}]^\perp}\Big)\notag
   	\end{align}
   	with
   	the orthogonal decomposition in the tangential plane and in the normal direction
   	\begin{align}
   	\mathcal{K}_{\overline{R}_{0}}^{\rm plate}=[\mathcal{K}_{\overline{R}_{0}}^{\rm plate}]^\parallel+[\mathcal{K}_{\overline{R}_{0}}^{\rm plate}]^\perp, \qquad  [\mathcal{K}_{\overline{R}_{0}}^{\rm plate}]^\parallel\coloneqq e_3\otimes e_3 \,\mathcal{K}_{\overline{R}_{0}}^{\rm plate}, \qquad [\mathcal{K}_{\overline{R}_{0}}^{\rm plate}]^\perp\coloneqq(\id_3-e_3\otimes e_3) \,\mathcal{K}_{\overline{R}_{0}}^{\rm plate}\,.
   	\end{align}
   \end{theorem}
   
   \begin{proof}
 Let us define $\Gamma_0=(\Gamma_1^0\,|\,\Gamma_2^0\,|\,\Gamma_3^0):=\Big(\mathrm{axl}(\overline{R}_{0}^T\,\partial_{\eta_1} \overline{R}_{0})\,|\, \mathrm{axl}(\overline{R}_{0}^T\,\partial_{\eta_2} \overline{R}_{0})\,|\,\axl \,(A)\,\Big)$. Then the homogenized curvature energy turns out to be 
  \begin{align}
  W_{\text{curv}}^{\text{hom}}((\Gamma_1^0\,|\,\Gamma_2^0))=\widetilde{W}_{\text{curv}}(\Gamma_1^0\,|\,\Gamma_2^0|c^*)=\inf_{c\in\R^3 } W_{\text{curv}}((\Gamma_1^0\,|\,\Gamma_2^0|c))\,.
  \end{align}
  
  By using the relation (\ref{curvenerg}), we start to do the calculations for the $\sym$, $\skw$ and trace parts as
  \begin{align}
  \sym\Gamma^0 =\begin{pmatrix}
  \Gamma  _{11}^0&\frac{\Gamma  _{12}^0+\Gamma  _{21}^0}{2}&\frac{c_1+\Gamma _{31}^0 }{2}\\\frac{\Gamma  _{21}^0+\Gamma  _{12}^0}{2}&\Gamma  _{22}^0&\frac{c_2+\Gamma _{32}^0 }{2}\\\frac{\Gamma  _{31}^0+c_1}{2}&\frac{\Gamma  _{32}^0+c_2}{2}&c_3
  \end{pmatrix}\,,\qquad \skw\,\Gamma^0 =\begin{pmatrix}
  0&\frac{\Gamma  _{12}^0-\Gamma  _{21}^0}{2}&\frac{c_1-\Gamma _{31}^0 }{2}\\\frac{\Gamma  _{21}^0-\Gamma  _{12}^0}{2}&0&\frac{c_2-\Gamma _{32}^0 }{2}\\\frac{\Gamma  _{31}^0-c_1}{2}&\frac{\Gamma  _{32}^0-c_2}{2}&0
  \end{pmatrix}\,,
  \end{align}
  and
$
  \tr(\Gamma_0 )=(\Gamma^0 _{11} +\Gamma^0 _{22} +c_3)\,.
$
  We have
  \begin{align}
  W_{\rm{curv}}(\Gamma_0  )&=\mu L_c^2\Big(b_1\big((\Gamma^0 _{11})^{2}+\frac{1}{2}(\Gamma _{12}^0 +\Gamma^0 _{21} )^2+\frac{1}{2}(c_1+\Gamma _{31}^{0})^2+(\Gamma _{22}^{0})^2+\frac{1}{2}(c_2+\Gamma _{32}^0 )^2+c_3^2\big)\\
  \nonumber&\qquad\ \ \ \ +b_2\big(\frac{1}{2}(\Gamma _{12}^{0}-\Gamma _{21}^{0})^2+\frac{1}{2}(c_1-\Gamma _{31}^0 )^2+\frac{1}{2}(c_2-\Gamma _{32}^0 )^2\big)+b_3(\Gamma _{11}^0 +\Gamma _{22}^0 +c_3)^2\Big)\,.
  \end{align}
 But  this is an easy optimization problem in $\mathbb{R}^3$. Indeed, the stationary points are  
  \begin{align}
  \nonumber 0&=\frac{\partial W_{\rm{curv}}(\Gamma  _0)}{\partial c_1}=b_1(c_1+\Gamma _{31}^0 )+b_2(c_1-\Gamma _{31}^0 )=(b_1+b_2)c_1+(b_1-b_2)\Gamma _{31}^0 \quad\Rightarrow \quad c_1=\frac{b_2-b_1}{b_1+b_2}\Gamma _{31}^0 \,,\\
  0&=\frac{\partial W_{\rm{curv}}(\Gamma _0 )}{\partial c_2}=b_1(c_2+\Gamma _{32}^0 )+b_2(c_2-\Gamma _{32}^0 )=(b_1+b_2)c_2+(b_1-b_2)\Gamma _{32}^0 \quad\Rightarrow \quad c_2=\frac{b_2-b_1}{b_1+b_2}\Gamma _{32} ^0\,,\\
  \nonumber 0&=\frac{\partial W_{\rm{curv}}(\Gamma  _0)}{\partial c_3}=b_1c_3+b_3(\Gamma _{11}^0 +\Gamma _{22}^0 +c_3)\quad\Rightarrow \quad c_3=\frac{-b_3}{b_1+b_3}(\Gamma _{11}^0 +\Gamma _{22}^0 )\,,
  \end{align}
and  the matrix defining the quadratic function in $c_1,c_2,c_3$ is positive definite, this stationary point is the minimizer, too.
  By inserting the unknowns inside $W_{\rm{curv}}$ we find $W_{\text{curv}}^{\text{hom, plate}}$ given by
  \begin{align}
  W_{\text{curv}}^{\text{hom, plate}}(\Gamma  )&=\mu L_c^2\Big(b_1\big((\Gamma _{11}^{0})^{2}+(\Gamma _{22}^{0})^{2}+(\frac{-b_3}{b_1+b_3}(\Gamma _{11}^0 +\Gamma _{22}^0 ))^2+\frac{1}{2}(\Gamma _{21}^0 +\Gamma _{12}^0 )^2+\frac{1}{2}(\frac{b_2-b_1}{b_1+b_2}\Gamma _{31}^0 +\Gamma _{31}^0 )^2\notag\\
  &\qquad\qquad +\frac{1}{2}(\frac{b_2-b_1}{b_1+b_2}\Gamma _{32}^0 +\Gamma _{32}^0 )^2\big)+b_2\big(\frac{1}{2}(\Gamma _{12}^0 -\Gamma _{21}^0 )^2+\frac{1}{2}(\frac{b_2-b_1}{b_1+b_2}\Gamma  _{31}^0-\Gamma _{31}^0 )^2\nonumber\\
  \nonumber&\qquad\qquad+\frac{1}{2}(\frac{b_2-b_1}{b_1+b_2}\Gamma _{32}^0 -\Gamma _{32}^0 )^2\big)+b_3\big((\Gamma _{11}^0 +\Gamma _{22}^0 )-\frac{b_3}{b_1+b_3}(\Gamma _{11}^0 +\Gamma _{22}^0 )\big)^2\Big)\\
  \nonumber&=\mu L_c^2\Big(b_1\big((\Gamma _{11}^0)^{2}+(\Gamma _{22}^{0})^{2}+\frac{b_3^2}{(b_1+b_3)^2}(\Gamma _{11}^0 +\Gamma _{22}^0 )^2+\frac{1}{2}(\Gamma _{21}^0 +\Gamma _{12}^0 )^2+2\frac{b_2^2}{(b_1+b_2)^2}(\Gamma _{31}^0)^{2}\\
  &\qquad\qquad +2\frac{b_2^2}{(b_1+b_2)^2}(\Gamma _{32}^0)^{2}\big)+b_2\big(\frac{1}{2}(\Gamma _{12}^0 -\Gamma _{21}^0 )^2+2\frac{b_1^2}{(b_1+b_2)^2}(\Gamma ^0)^{2}_{31}+2\frac{b_1^2}{(b_1+b_2)^2}(\Gamma _{32}^0)^{2}\big)\nonumber\\
  &\qquad\qquad+b_3\frac{b_1^2}{(b_1+b_3)^2}(\Gamma _{11}^0 +\Gamma _{22}^0 )^2\Big)\\
  \nonumber&=\mu L_c^2\Big(b_1\big((\Gamma _{11}^0)^{2}+(\Gamma _{22}^{0})^ {2}\big)+\frac{b_1b_3}{(b_1+b_3)}(\Gamma _{11}^0 +\Gamma _{22}^0 )^2+\frac{b_1}{2}(\Gamma _{21}^0 +\Gamma _{12}^0 )^2\\
  \nonumber &\quad\qquad\qquad+2\frac{b_1b_2}{(b_1+b_2)}(\Gamma _{31}^{0})^2+2\frac{b_1b_2}{(b_1+b_2)}(\Gamma _{32}^0)^2+\frac{b_2}{2}(\Gamma _{21}^0 -\Gamma _{12}^0 )^2\Big)\\
  \nonumber&=\mu L_c^2\Big(b_1\norm{\sym \Gamma _\square }^2+b_2\norm{\skew \Gamma _\square }^2+\frac{b_1b_3}{(b_1+b_3)}\tr(\Gamma _\square )^2+\frac{2b_1b_2}{(b_1+b_2)}\norm{\matr{\Gamma _{31}^{0}\\\Gamma _{32}^{0}}}^2\Big)\,,
  \end{align}
  where $\Gamma _\square =\matr{\Gamma _{11}^0 &\Gamma _{12}^0 \\\Gamma _{21}^0 &\Gamma _{22}^0 }$. 
  
  Therefore, the homogenized curvature energy for the flat Cosserat-shell model  is
  \begin{align}\label{homflat}
  W_{\text{curv}}^{\text{hom, plate}}(\Gamma  )&=\mu L_c^2\Big(b_1\norm{\sym \Gamma _\square }^2+b_2\norm{\skew \Gamma _\square }^2+\frac{b_1b_3}{(b_1+b_3)}\tr(\Gamma _\square )^2+\frac{2b_1b_2}{(b_1+b_2)}\norm{\matr{\Gamma _{31}^{0}\\\Gamma _{32}^{0}}}^2\Big)\,\\
  &=\mu L_c^2\Big(b_1\norm{\sym[\mathcal{K}_{\overline{R}_{0}}^{\rm plate}]^\parallel}^2+b_2\norm{\skew [\mathcal{K}_{\overline{R}_{0}}^{\rm plate}]^\parallel}^2+\frac{b_1b_3}{(b_1+b_3)}\tr([\mathcal{K}_{\overline{R}_{0}}^{\rm plate}]^\parallel)^2+\frac{2\,b_1b_2}{b_1+b_2}\norm{[\mathcal{K}_{\overline{R}_{0}}^{\rm plate}]^\perp}\Big)\notag
  \end{align}
  with
  the orthogonal decomposition in the tangential plane and in the normal direction
  \begin{align*}
 \hspace{1.5cm} \mathcal{K}_{\overline{R}_{0}}^{\rm plate}=[\mathcal{K}_{\overline{R}_{0}}^{\rm plate}]^\parallel+[\mathcal{K}_{\overline{R}_{0}}^{\rm plate}]^\perp, \qquad [\mathcal{K}_{\overline{R}_{0}}^{\rm plate}]^\parallel\coloneqq e_3\otimes e_3 \,\mathcal{K}_{\overline{R}_{0}}^{\rm plate}, \qquad  [\mathcal{K}_{\overline{R}_{0}}^{\rm plate}]^\perp\coloneqq(\id_3-e_3\otimes e_3) \,\mathcal{K}_{\overline{R}_{0}}^{\rm plate}\,.\hspace{1.5cm} \qedhere
  \end{align*}
  \end{proof}
  Since we have now the explicit form of both homogenized energies (membrane and curvature),  we are  ready to indicate the exact form of the $\Gamma$-limit of the sequence of functionals $\mathcal{J}_{h_j}\col X\rightarrow \overline{\mathbb{R}}$  and to provide the following  theorem, see \cite{Maryam}
  \begin{theorem}\label{Theorem_homo_plate}
  	Assume   the boundary data satisfy the conditions
  	\begin{equation}\label{2n5}
  	\varphi^\natural_d=\varphi_d\big|_{\partial \Omega_1} \text{(in the sense of traces) for } \ \varphi_d\in {\rm H}^1(\Omega_1;\mathbb{R}^3), \qquad \Gamma_1\subset \partial
  	\end{equation}
  and let the constitutive parameters satisfy 
  	\begin{align}
  	\mu\,>0, \qquad\quad \kappa>0, \qquad\quad \mu_{\rm c}> 0,\qquad\quad a_1>0,\quad\quad a_2>0,\qquad\quad a_3>0\,.
  	\end{align} 
  	Then, 
  	for any sequence $(\varphi_{h_j}^\natural, \overline{R}_{h_j}^{\natural})\in X$ such that $(\varphi_{h_j}^\natural, \overline{R}_{h_j}^{\natural})\rightarrow(\varphi_0, \overline{R}_{0})$ as $h_j\to 0$,	
  	the sequence of functionals $\mathcal{I}_{h_j}\col X\rightarrow \overline{\mathbb{R}}$  from \eqref{energyinnewarea 1}
  	\textit{ $\Gamma$-converges} to the limit energy  functional $\mathcal{I}_0\col X\rightarrow \overline{\mathbb{R}}$ defined by
  	\begin{equation}\label{couple}
  	\mathcal{I}_0 (m,\overline{R}_{0}) =
  	\begin{cases}\dd
  	\int_{\omega} [ {W}_{\rm mp}^{{\rm hom, plate}}(\mathcal{E}_{m,\overline{R}_{0}}^{\rm plate})+\widetilde{W}_{\rm curv}^{{\rm hom, plate}}(\mathcal{K}_{\overline{R}_{0}}^{\rm plate})]\; d\omega\quad &\text{if}\quad (m,\overline{R}_{0})\in \mathcal{S}_\omega' \,,
  	\\
  	+\infty\qquad &\text{else in}\; X,
  	\end{cases}
  	\end{equation}
  	where 
  	\begin{align}\label{approxi}
  	m(x_1,x_2)&:=\varphi_0 (x_1,x_2)=\lim_{h_j\to 0} \varphi_{h_j}^\natural(x_1,x_2,\frac{1}{h_j}x_3), \qquad \overline{Q}_{e,0}(x_1,x_2)=\lim_{h_j\to 0} \overline{R}_{h_j}^\natural(x_1,x_2,\frac{1}{h_j}x_3),\notag\\
  	\mathcal{E}_{m,\overline{R}_{0}}^{\rm plate}&=\overline{R}^{T}_{0}(\D  m|0)-\id_2^\flat\,,\qquad
  	\mathcal{K}_{\overline{R}_{0}}^{\rm plate} =\Big(\mathrm{axl}(\overline{R}^{T}_{0}\,\partial_{x_1} \overline{R}_{0})\,|\, \mathrm{axl}(\overline{R}^{T}_{0}\,\partial_{x_2} \overline{R}_{0})\,|0\Big)\not\in {\rm Sym}(3)\,,\notag
  	\end{align}
  	and
  	\begin{align}
  	\nonumber {W}_{\rm mp}^{{\rm hom, plate}}
  	(\mathcal{E}_{m,\overline{R}_{0}}^{\rm plate})&=
  	\, \mu\,\lVert  \mathrm{sym}\,   \,[\mathcal{E}_{m,\overline{R}_{0} }^{\rm plate}]^{\parallel}\rVert^2 +  \mu_{\rm c}\,\lVert \mathrm{skew}\,   \,[\mathcal{E}_{m,\overline{R}_{0} }^{\rm plate}]^{\parallel}\rVert^2 +\,\dfrac{\lambda\,\mu\,}{\lambda+2\,\mu\,}\,\big[ \mathrm{tr}    ([\mathcal{E}_{m,\overline{R}_{0} }^{\rm plate}]^{\parallel})\big]^2 +\frac{2\,\mu\, \,  \mu_{\rm c}}{\mu_c\,+\mu\,}\norm{[\mathcal{E}_{m,\overline{R}_{0} }^{\rm plate}]^Tn_0}^2\\
  	&=W_{\mathrm{shell}}\big(   [\mathcal{E}_{m,\overline{R}_{0} }^{\rm plate}]^{\parallel} \big)+\frac{2\,\mu\, \,  \mu_{\rm c}}{\mu_c\,+\mu\,}\lVert [\mathcal{E}_{m,\overline{R}_{0} }^{\rm plate}]^{\perp}\rVert^2,\\  \widetilde{W}^{{\rm hom, plate}}_{\rm curv}(\mathcal{K}_{\overline{R}_{0}}^{\rm plate})
  	\nonumber&=\inf_{A\in \mathfrak{so}(3)}\widetilde{W}_{\rm curv}\Big(\mathrm{axl}(\overline{R}^{T}_{0}\,\partial_{\eta_1} \overline{R}_{0})\,|\, \mathrm{axl}(\overline{R}^{T}_{0}\,\partial_{\eta_2} \overline{R}_{0})\,|\,\axl(A)\,\Big)[(\D _x \Theta)^\natural(0)]^{-1}\\
  	\nonumber&=\mu L_c^2\Big(b_1\norm{\sym[\mathcal{K}_{\overline{R}_{0}}^{\rm plate}]^\parallel}^2+b_2\norm{\skew [\mathcal{K}_{\overline{R}_{0}}^{\rm plate}]^\parallel}^2+\frac{b_1b_3}{(b_1+b_3)}\tr([\mathcal{K}_{\overline{R}_{0}}^{\rm plate}]^\parallel)^2+\frac{2\,b_1b_2}{b_1+b_2}\norm{[\mathcal{K}_{\overline{R}_{0}}^{\rm plate}]^\perp}\Big)\,.
  	\end{align}
  \end{theorem}

  \section{Homogenized curvature energy for the curved Cosserat-shell model via $\Gamma$-convergence}\label{curved}\setcounter{equation}{0}
  
  In this section we consider the case of a curved Cosserat-shell model and we give the explicit form and the detailed calculation  of the homogenized curvature energy. In comparison to the flat Cosserat-shell model, the calculations are more  complicated.
  Hence, let us consider an elastic material which  in its reference configuration fills the three dimensional \textit{shell-like thin} domain $\Omega_\xi\subset R ^3$, i.e., we  assume that there exists a $C^1$-diffeomorphism $\Theta\col R ^3\rightarrow  R ^3$ with $\Theta(x_1,x_2,x_3):=(\xi_1,\xi_2,\xi_3)$ such that $\Theta(\Omega_h)=\Omega_\xi$ and $\omega_\xi=\Theta(\omega\times\{0\})$, where  $\Omega_h\subset  R ^3$ for $\Omega_h=\omega\times\big[-\frac{h}{2},\frac{h}{2}\big]$, with $\omega\subset  R ^2$  a bounded domain with Lipschitz boundary $\partial\omega$. The scalar $0<h\ll1$ is called \textit{thickness} of the shell, while the domain $\Omega_h$ is called \textit{fictitious Cartesian configuration} of the body. In fact, in this paper,  we consider the following diffeomorphism $\Theta\col R ^3\rightarrow  R ^3$ which describes the curved surface of the shell
  \begin{align}
  \Theta(x_1,x_2,x_3)=y_0(x_1,x_2)+x_3\,n_0(x_1,x_2)\,,
  \end{align}
  where $y_0\col\omega\rightarrow  R ^3$ is a $C^2(\omega)$-function  and $n_0=\frac{\partial_{x_1}y_0\times \partial_{x_2}y_0}{\norm{\partial_{x_1}y_0\times \partial_{x_2}y_0}}$ is the unit normal vector on $\omega_\xi$. 
  Remark that
  \begin{align}
  \D  _x \Theta(x_3)&\,=\,(\D  y_0|n_0)+x_3(\D  n_0|0) \, \  \forall\, x_3\in \left(-\frac{h}{2},\frac{h}{2}\right),
  \ \ 
  \D  _x \Theta(0)\,=\,(\D  y_0|\,n_0),\ \ [\D  _x \Theta(0)]^{-T}\, e_3\,=n_0,
  \end{align}
  and  $\det\D  _x \Theta(0)=\det(\D  y_0|n_0)=\sqrt{\det[ (\D  y_0)^T\D  y_0]}$ represents the surface element.  We also have the polar decomposition $\D _x\Theta(0)=Q_0\,U_0$, where
  \begin{align}
  Q_0=\text{polar}(\D _x\Theta(0))=\text{polar}([\D _x\Theta(0)]^{-T})\in \text{SO(3)}\quad\text{and}\quad U_0\in \text{Sym}^+(3)\,.
  \end{align}
  
  The first step in our  shell model is to transform the problem to a variational problem defined on the fictitious flat configuration $\Omega_h=\omega\times\big[-\frac{h}{2},\frac{h}{2}\big]$. 
  The next step, in order to apply the $\Gamma$-convergence technique, is to transform the minimization problem onto the \textit{fixed domain} $\Omega_1$, which is independent of the thickness $h$. These two steps were done in \cite{Maryam}, the three-dimensional problem \eqref{energy model} (corresponding to the Cosserat-curvature tensor $
 \alpha$) being equivalent to the following minimization problem on $\Omega_1$
  \begin{align}\label{energy model.fictituos.ond omega1}
  \nonumber &I^\natural_h(\varphi^\natural,\D _\eta^h\varphi^\natural,\overline{Q}_{e,h}^\natural,\Gamma_{e,h}^\natural)
  =\int_{\Omega_1} \Big({W}_{\rm mp}(U_{e,h}^\natural)+{\widetilde{W}}_{\rm curv}(\Gamma_{e,h}^\natural)\Big)\det(\D _\eta\zeta(\eta))\det((\D _x\Theta)^\natural(\eta_3) )\;dV_\eta\\
  & =\int_{\Omega_1}\; h \;\Big[\Big({{W}}_{\rm mp}(U_{e,h}^\natural)+{\widetilde{W}}_{\rm curv}(\Gamma_{e,h}^\natural)\Big)\det((\D _x\Theta)^\natural (\eta_3))\Big]\;dV_\eta
  \mapsto \min\;\text{w.r.t}\; (\varphi^\natural,\overline{Q}_{e,h}^\natural)\,,
  \end{align}
  where 
  \begin{align}\label{boun co}
  \nonumber {{W}}_{\rm mp}(U_{e,h}^\natural)&=\;\mu\,\norm{\sym (U_{e,h}^\natural-\id_3)}^2+\mu_c\,\norm{\skew (U_{e,h}^\natural-\id_3)}^2 + \frac{\lambda}{2}[\tr(\sym(U_{e,h}^\natural-\id_3))]^2\,,\\
  \widetilde{W}_{\rm curv}(\Gamma_{e,h}^\natural)&\,=\,\mu\, {L}_{\rm c}^2 \left( b_1\,\lVert \,\text{sym} \,\Gamma_{e,h}^\natural\rVert^2+b_2\,\lVert \text{skew} \,\Gamma_{e,h}^\natural\rVert^2+\,b_3\,
  [{\rm tr}(\Gamma_{e,h}^\natural)]^2\right)\,,
  \\
U_{e,h}^\natural &=\overline{Q}_{e,h}^{\natural,T}F^\natural_h [(\D _x\Theta)^\natural (\eta_3)]^{-1}=\overline{Q}_{e,h}^{\natural,T}\D _\eta^h\varphi^\natural(\eta)[(\D _x\Theta)^\natural(\eta_3) ]^{-1}\,,\notag\\
   \Gamma_{e,h}^\natural& =\Big(\mathrm{axl}(\overline{Q}_{e,h}^{\natural,T}\,\partial_{\eta_1} \overline{Q}_{e,h}^\natural)\,|\, \mathrm{axl}(\overline{Q}_{e,h}^{\natural,T}\,\partial_{\eta_2} \overline{Q}_{e,h}^\natural)\,|\,\frac{1}{h}\mathrm{axl}(\overline{Q}_{e,h}^{\natural,T}\,\partial_{\eta_3} \overline{Q}_{e,h}^\natural)\,\Big)[(\D _x\Theta)^\natural (\eta_3)]^{-1},\notag
  \end{align}
  with $(\D _x\Theta)^\natural(\eta_3) $ the nonlinear scaling (see \eqref{nonlinear sca}) of $\D _x\Theta$, $F^\natural_h=\D _\eta^h\varphi^\natural$  the nonlinear scaling of the gradient of the mapping 
  $
  \varphi\col \Omega_h\rightarrow \Omega_c\,,\  \varphi(x_1,x_2,x_3)=\varphi_\xi(\Theta(x_1,x_2,x_3))\,
  $, $\overline{Q}_{e,h}^\natural$ the nonlinear scaling of  the \textit{elastic microrotation}  $\overline{Q}_{e}\col\Omega_h\rightarrow \text{SO(3)}$  defined by
  $
  \overline{Q}_{e}(x_1,x_2,x_3):=\overline{R}_\xi(\Theta(x_1,x_2,x_3))\,. 
  $  Since for $\eta_3=0$ the values of $\D _x \Theta$, $Q_0$, $U_0$  expressed in terms of $(\eta_1, \eta_2,0)$ and $(x_1,x_2,0)$ coincide, we will omit the sign $\cdot ^\natural$ and we will understand from the context the variables into discussion, i.e.,  
  \begin{align}
  (\D _x \Theta)(0)&:=(\D  y_0\,|n_0)=(\D _x \Theta)^\natural(\eta_1,\eta_2,0)\equiv (\D _x \Theta)(x_1,x_2,0), \\ Q_0(0)&:=Q_0^\natural(\eta_1,\eta_2,0)\equiv Q_0(x_1,x_2,0), \qquad\qquad U_0(0):=U_0^\natural(\eta_1,\eta_2,0)\equiv U_0(x_1,x_2,0).\notag
  \end{align}

  In order to construct the $\Gamma$-limit of the rescaled energies
 \begin{align}\label{energyinnewarea}
 \mathcal{I}_h^\natural(\varphi^\natural,  \D ^h_\eta \varphi^\natural,  \overline{R}_h^\natural,\Gamma_h^\natural) 
 &=
 \begin{cases}
 \displaystyle\frac{1}{h}\,I^\natural_h(\varphi^\natural,\D _\eta^h\varphi^\natural,\overline{Q}_{e,h}^\natural,\Gamma_{e,h}^\natural)\quad &\text{if}\;\; (\varphi^\natural,  \overline{Q}_{e,h}^\natural)\in \mathcal{S}',
 \\
 +\infty\qquad &\text{else in}\; X,
 \end{cases}
 \end{align}
 for curved Cosserat-shell model we have to solve the following {\bf four} {\bf (not only two as for flat Cosserat-shell models)} auxiliary optimization problem.
\begin{itemize}
	\item[O1:] For  each $\varphi^\natural:\Omega_1\rightarrow \mathbb{R}^3$ and $\overline{Q}_{e,h}^\natural:\Omega_1\rightarrow{\rm SO}(3)$ we determine a vector $d^*\in \R^3$ through
	\begin{align}\label{hom inf}
	{W}_{\rm mp}^{{\rm hom},\natural}(  \mathcal{E}_{\varphi^\natural,\overline{Q}_{e,h}^{\natural} }):&= {W}_{\rm mp}\Big(\overline{Q}_{e,h}^{\natural,T}(\D _{(\eta_1,\eta_2)} \varphi^\natural|d^*)[(\D _x\Theta)^\natural (\eta_3)]^{-1}\Big)\notag\\&=\inf_{c\in \mathbb{R}^3} {W}_{\rm mp}\Big(\overline{Q}_{e,h}^{\natural,T}(\D _{(\eta_1,\eta_2)} \varphi^\natural|c)[(\D _x\Theta)^\natural (\eta_3)]^{-1}\Big),
	\end{align}where
$
	\mathcal{E}_{\varphi^\natural,\overline{Q}_{e,h}^{\natural} }:=(\overline{Q}_{e}^{\natural, T}\D _{(\eta_1,\eta_2)} \varphi^\natural-(\D  y_0)^\natural |0)[(\D _x\Theta)^\natural (\eta_3)]^{-1}\,
	$
	represents the non fully\footnote{Here, "non-fully" means that the introduced quantities still depend on $\eta_3$ and $h$, because the elements $\D _{(\eta_1,\eta_2)}\varphi^\natural$ still depend on $\eta_3$ and $\overline{ Q}^{\natural , T}$ depends on $h$.} dimensional reduced elastic shell strain  tensor.
	
	\item[O2:] For  each pair $(m, \overline{Q}_{e,0})$, where $m:\omega\to \mathbb{R}^3$,  $\overline{Q}_{e,0}:\omega\to {\rm SO}(3)$ we determine the vector $\widetilde{d}^*\in\mathbb{R}^3$  through
	\begin{align}\label{O2}
	{W}_{\rm mp}^{\rm hom}
	(\mathcal{E}_{m,\overline{Q}_{e,0}}):={W}_{\rm mp}\Big(\overline{Q}^{T}_{e,0}(\D  m|\widetilde{d}^*)[(\D _x\Theta)(0)]^{-1}\Big)&=\inf_{\widetilde{c}\in \mathbb{R}^3} {W}_{\rm mp}\Big(\overline{Q}^{T}_{e,0}(\D  m|\widetilde{c})[(\D _x\Theta)(0)]^{-1}\Big)\\&=\inf_{\widetilde{c}\in \mathbb{R}^3} {W}_{\rm mp}\Big(\mathcal{E}_{m,\overline{Q}_{e,0}}-(0|0|\widetilde{c})[(\D _x\Theta)(0)]^{-1}\Big),\notag
	\end{align}
	where 
	$
	\mathcal{E}_{m,\overline{Q}_{e,0}} :=(\overline{Q}^{T}_{e,0}\D  m-\D  y_0 |0)[\D _x\Theta(0)]^{-1}\,
$
	represents the \textit{elastic shell strain tensor}. 
	
	\item[O3:] For  each  $\overline{Q}_{e,h}^\natural:\Omega_1\rightarrow{\rm SO}(3)$ we determine the skew-symmetric matrix $A^*\in \so{3}$, i.e. its axial vector $\axl A^*\in \mathbb{R}^3$ through 
	\begin{align}\label{inf curv}
	 \widetilde{W}^{{\rm hom},\natural}_{\rm curv}( \mathcal{K}_{\overline{Q}_{e,h}^\natural}):&=\widetilde{W}_{\rm curv}\Big(\big(\mathrm{axl}(\overline{Q}_{e,h}^{\natural,T}\,\partial_{\eta_1} \overline{Q}_{e,h}^\natural)\,|\, \mathrm{axl}(\overline{Q}_{e,h}^{\natural,T}\,\partial_{\eta_2} \overline{Q}_{e,h}^\natural)\,|\,\axl\,(A^*)\,\big)[(\D _x\Theta)^\natural (\eta_3)]^{-1}\Big)\\
	&=\inf_{A\in \mathfrak{so}(3)}\widetilde{W}_{\rm curv}\Big(\big(\mathrm{axl}(\overline{Q}_{e,h}^{\natural,T}\,\partial_{\eta_1} \overline{Q}_{e,h}^\natural)\,|\, \mathrm{axl}(\overline{Q}_{e,h}^{\natural,T}\,\partial_{\eta_2} \overline{Q}_{e,h}^\natural)\,|\,\axl\,(A)\,\big)[(\D _x\Theta)^\natural (\eta_3)]^{-1}\Big),\nonumber
	\end{align}
	where
	$
	\mathcal{K}_{\overline{Q}_{e,h}^\natural}  :\,=\,  \Big(\mathrm{axl}(\overline{Q}_{e,h}^{\natural,T}\,\partial_{\eta_1} \overline{Q}_{e,h}^\natural)\,|\, \mathrm{axl}(\overline{Q}_{e,h}^{\natural,T}\,\partial_{\eta_2} \overline{Q}_{e,h}^\natural)\,|0\Big)[(\D _x\Theta)^\natural (\eta_3)]^{-1}\notag\,
$
	represents a not fully reduced elastic shell bending-curvature tensor, in the sense that it still depends on $\eta_3$ and $h$, since  $\overline{Q}_{e,h}^\natural=\overline{Q}_{e,h}^\natural(\eta_1,\eta_2,\eta_3)$. Therefore, $\widetilde{W}_{\rm curv}^{{\rm hom},\natural}( \mathcal{K}_{\overline{Q}_{e,h}^\natural})$ given by the above definitions still depends on $\eta_3$ and $h$.
	\item[O4:] For  each  $\overline{Q}_{e,0} :\Omega_1\rightarrow{\rm SO}(3)$ we determine the skew-symmetric matrix $A^*\in \so(3)$, i.e. its axial vector $\axl A^*\in \mathbb{R}^3$, though

	\begin{align}\label{inf curv0} \widetilde{W}^{{\rm hom}}_{\rm curv}(\mathcal{K}_{\overline{Q}_{e,0}}):&=\widetilde{W}_{\rm curv}^*\Big(\big(\mathrm{axl}(\overline{Q}^{T}_{e,0}\,\partial_{\eta_1} \overline{Q}_{e,0})\,|\, \mathrm{axl}(\overline{Q}^{T}_{e,0}\,\partial_{\eta_2} \overline{Q}_{e,0})\,|\,\axl\,(A^*)\,\big)[(\D _x \Theta)^\natural(0)]^{-1}\Big)\\
	&=\inf_{A\in \mathfrak{so}(3)}\widetilde{W}_{\rm curv}\Big(\big(\mathrm{axl}(\overline{Q}^{T}_{e,0}\,\partial_{\eta_1} \overline{Q}_{e,0})\,|\, \mathrm{axl}(\overline{Q}^{T}_{e,0}\,\partial_{\eta_2} \overline{Q}_{e,0})\,|\,\axl\,(A)\,\big)[(\D _x \Theta)^\natural(0)]^{-1}\Big)\,,\notag
	\end{align}
	where
	$
	\mathcal{K}_{\overline{Q}_{e,0}}  :\,=\,  \Big(\mathrm{axl}(\overline{Q}^{T}_{e,0}\,\partial_{x_1} \overline{Q}_{e,0})\,|\, \mathrm{axl}(\overline{Q}^{T}_{e,0}\,\partial_{x_2} \overline{Q}_{e,0})\,|0\Big)[\D _x\Theta (0)\,]^{-1}\not\in {\rm Sym}(3) 
	$
	represents the {\it  
	the elastic shell bending-curvature tensor}.
\end{itemize}
 
 Let us remark that having the solutions of the optimization problems O1 and O3, the solutions for the optimization problems O2 and O4, respectively, follow immediately. However, we cannot skip the solution of the optimization problems O1 and O3 and use only the solutions of the optimization problems O2 and O4, since the knowledge of $W_{\rm mp}^{\rm hom}$ and  $W_{\rm curv}^{\rm hom}$ is important in the proof of the $\Gamma$-convergence result. This is the first major difference between $\Gamma$-convergence for curved initial configurations and flat initial configuration.
 
 The solutions of the first two optimization problems and the complete calculations where given in \cite{Maryam}, while the analytical calculations of the last two optimization problems were left open until now. 
 
 For the completeness of the exposition we recall the following result
 \begin{theorem}
 	\cite{Maryam}
 The solution of the optimization problem O2 is
 \begin{align}\label{formula d}
 \widetilde{d}^*&=\Big(1-\frac{\lambda}{2\,\mu\,+\lambda}\iprod{  \mathcal{E}_{m,\overline{Q}_{e,0}} ,\id_3}\Big) \overline{Q}_{e,0} n_0+\frac{\mu_c\,-\mu\,}{\mu_c\,+\mu\,}\;\overline{Q}_{e,0} \mathcal{E}_{m,\overline{Q}_{e,0}}^T n_0\,,
 \end{align}
 and
 \begin{align}\label{hominf_2D}
 {W}_{\rm mp}^{\rm hom}
 (\mathcal{E}_{m,\overline{Q}_{e,0}})&=
 \, \mu\,\lVert  \mathrm{sym}\,   \,\mathcal{E}_{m,\overline{Q}_{e,0}}^{\parallel}\rVert^2 +  \mu_{\rm c}\,\lVert \mathrm{skew}\,   \,\mathcal{E}_{m,\overline{Q}_{e,0}}^{\parallel}\rVert^2 +\,\dfrac{\lambda\,\mu\,}{\lambda+2\,\mu\,}\,\big[ \mathrm{tr}    (\mathcal{E}_{m,\overline{Q}_{e,0}}^{\parallel})\big]^2 +\frac{2\,\mu\, \,  \mu_{\rm c}}{\mu_c\,+\mu\,}\norm{\mathcal{E}_{m,\overline{Q}_{e,0}}^Tn_0}^2\\
 &=W_{\mathrm{shell}}\big(   \mathcal{E}_{m,\overline{Q}_{e,0}}^{\parallel} \big)+\frac{2\,\mu\, \,  \mu_{\rm c}}{\mu_c\,+\mu\,}\lVert \mathcal{E}_{m,\overline{Q}_{e,0}}^{\perp}\rVert^2,\notag
 \end{align}
 where 
$
 W_{\mathrm{shell}}\big(   \mathcal{E}_{m,\overline{Q}_{e,0}}^{\parallel} \big)=
 \, \mu\,\lVert  \mathrm{sym}\,   \,\mathcal{E}_{m,\overline{Q}_{e,0}}^{\parallel}\rVert^2 +  \mu_{\rm c}\,\lVert \mathrm{skew}\,   \,\mathcal{E}_{m,\overline{Q}_{e,0}}^{\parallel}\rVert^2 +\,\dfrac{\lambda\,\mu\,}{\lambda+2\,\mu\,}\,\big[ \mathrm{tr}    (\mathcal{E}_{m,\overline{Q}_{e,0}}^{\parallel})\big]^2
$
 with  the orthogonal decomposition in the tangential plane and in the normal direction
 \begin{align}\label{orthog decom}
 \mathcal{E}_{m,\overline{Q}_{e,0}}=\mathcal{E}_{m,\overline{Q}_{e,0}}^\parallel+\mathcal{E}_{m,\overline{Q}_{e,0}}^\perp, \quad\qquad  \mathcal{E}_{m,\overline{Q}_{e,0}}^\parallel\coloneqq{\rm A}_{y_0} \,\mathcal{E}_{m,\overline{Q}_{e,0}}, \quad \qquad \mathcal{E}_{m,\overline{Q}_{e,0}}^\perp\coloneqq(\id_3-{\rm A}_{y_0}) \,\mathcal{E}_{m,\overline{Q}_{e,0}},
 \end{align}
and
$
{\rm A}_{y_0}:=(\D  y_0|0)\,\,[\D \Theta_x(0) \,]^{-1}\in\mathbb{R}^{3\times 3}.
$
\end{theorem}
 In the remainder of this section we provide the explicit solutions of the optimization problems O3 and O4. We remark that, while in the case of flat initial configuration the solution of O4 is very easy to be found, in the case of curved initial configuration the calculations are more difficult. Beside this, for curved initial configurations there is a need to solve the optimization problem O3, too. Notice that for flat initial configuration the optimization problems O3 and O4 coincide.
\subsection{The calculation of the homogenized curvature energy}
We have the following isotropic curvature energy formula for a curved configuration
\begin{align}
\widetilde{W}_{\text{curv}}(\Gamma_{e,h}^\natural )&=\mu L_c^2\Big(b_1\norm{\sym\, \Gamma_{e,h}^\natural }^2+b_2\,\norm{\skew\Gamma_{e,h}^\natural }^2+b_3\tr(\Gamma_{e,h}^\natural )^2\Big)\,.
\end{align}

\begin{theorem}The solution of the optimization problem O3 given by \eqref{inf curv} is
	\begin{align}c^*&=\frac{(b_2-b_1)}{b_1+b_2}\mathcal{K}_{\overline{Q}_{e,h}^\natural}^Tn_0-\frac{2b_3}{2(b_1+b_3)}\tr(\mathcal{K}_{\overline{Q}_{e,h}^\natural})n_0\end{align}
	and the coresponding homogenized curvature energy is
	\begin{align}\label{homocurv}
	\nonumber	W_{\text{curv}}^{\text{hom}}(\mathcal{K}_{\overline{Q}_{e,h}^\natural} )
	&=\mu L_c^2\Big(b_1\norm{\sym\mathcal{K}_{\overline{Q}_{e,h}^\natural} ^\parallel}^2+b_2\norm{\skew \mathcal{K}_{\overline{Q}_{e,h}^\natural} ^\parallel}^2+\frac{b_1b_3}{(b_1+b_3)}\tr(\mathcal{K}_{\overline{Q}_{e,h}^\natural} ^\parallel)^2+\frac{2b_1b_2}{b_1+b_2}\norm{\mathcal{K}_{\overline{Q}_{e,h}^\natural} ^\perp}\Big)\,,\notag
	\end{align}
	where $\mathcal{K}_{\overline{Q}_{e,h}^\natural} ^\parallel$ and $\mathcal{K}_{\overline{Q}_{e,h}^\natural} ^\perp$ represent the orthogonal decomposition of 
	the a not fully reduced elastic shell bending-curvature tensor $\mathcal{K}_{\overline{Q}_{e,h}^\natural}$ in the tangential plane and in the normal direction, respectively.
\end{theorem}
\begin{proof}
We need to find 
\begin{align}
\widetilde{W}_{\rm curv}^{{\rm hom},\natural}( \mathcal{K}_{\overline{Q}_{e,h}^\natural})&=\widetilde{W}_{\text{curv}}( \mathcal{K}_{\overline{Q}_{e,h}^\natural}+(0|0|c^*)[(\D  _x\Theta)^\natural(\eta_3)]^{-1})=\inf_{c\in \R^3} \widetilde{W}_{\text{curv}}(\underbrace{ \mathcal{K}_{\overline{Q}_{e,h}^\natural}+(0|0|c)[(\D  _x\Theta)^\natural(\eta_3)]^{-1}}_{=:\mathcal{K}_{\overline{Q}_{e,h}^\natural}^{c^*}})\,.
\end{align}
The Euler-Lagrange equations appear from variations with respect to arbitrary increments $\delta c\in \R^3$. 
\begin{align}
\iprod{D W_{\text{curv}}( \mathcal{K}_{\overline{Q}_{e,h}^\natural}^{c^*}),(0|0|\delta c)[(\D  _x\Theta)^\natural(\eta_3)]^{-1}}=0\,\quad&\Leftrightarrow\quad \iprod{[D \widetilde{W}_{\text{curv}}( \mathcal{K}_{\overline{Q}_{e,h}^\natural}^{c^*})]\, [(\D  _x\Theta)^\natural(\eta_3)]^{-T},e_3\otimes \delta c}=0\notag
\\
\quad&\Leftrightarrow\quad \iprod{[D \widetilde{W}_{\text{curv}}( \mathcal{K}_{\overline{Q}_{e,h}^\natural}^{c^*})]\, [(\D  _x\Theta)^\natural(\eta_3)]^{-T}\,e_3, \delta c}=0
\notag
\\
\quad&\Leftrightarrow\quad \iprod{[D \widetilde{W}_{\text{curv}}( \mathcal{K}_{\overline{Q}_{e,h}^\natural}^{c^*})]\, n_0, \delta c}=0\quad \forall \delta c\in \mathbb{R}^3.
\end{align}
Therefore, we deduce that if $c^*$ is a minimum then
\begin{align}\label{symskewtra curv}
&[D \widetilde{W}_{\text{curv}}( \mathcal{K}_{\overline{Q}_{e,h}^\natural}^{c^*})]\, n_0=0\quad \Leftrightarrow \quad
\Big(2a_1\sym (\mathcal{K}_{\overline{Q}_{e,h}^\natural}^{c^*})+2\,a_2\skew (\mathcal{K}_{\overline{Q}_{e,h}^\natural}^{c^*})+2a_3\tr(\mathcal{K}_{\overline{Q}_{e,h}^\natural}^{c^*})\Big)n_0=0\,.
\end{align}
Since 
$
\mathcal{K}_{\overline{Q}_{e,h}^\natural}^{c^*}=\mathcal{K}_{\overline{Q}_{e,h}^\natural}+(0|0|{c^*})[(\D  _x\Theta)^\natural (\eta_3)]^{-1}\,,
$ we have
\begin{align}
\nonumber 2\sym\big(\mathcal{K}_{\overline{Q}_{e,h}^\natural}^{c^*}\big)n_0&=2\Big(\sym(\mathcal{K}_{\overline{Q}_{e,h}^\natural})+\sym((0|0|{c^*})[(\D  _x\Theta)^\natural(\eta_3) ]^{-1})\Big)n_0\\
\nonumber&= \Big(\mathrm{axl}(\overline{Q}_{e,h}^{\natural,T}\,\partial_{\eta_1} \overline{Q}_{e,h}^\natural)\,|\, \mathrm{axl}(\overline{Q}_{e,h}^{\natural,T}\,\partial_{\eta_2} \overline{Q}_{e,h}^\natural)\,|0\Big)\underbrace{[(\D  _x\Theta)^\natural ]^{-1}(\eta_3)\,n_0}_{=e_3}+\mathcal{K}_{\overline{Q}_{e,h}^\natural}^Tn_0\\
\nonumber&\qquad+(0|0|{c^*})[(\D  _x\Theta)^\natural(\eta_3) ]^{-1}n_0+((0|0|{c^*})[(\D  _x\Theta)^\natural(\eta_3) ]^{-1})^Tn_0\\
&=\mathcal{K}_{\overline{Q}_{e,h}^\natural}^Tn_0+{c^*}+((0|0|{c^*})[(\D  _x\Theta)^\natural(\eta_3) ]^{-1})^Tn_0\,.
\end{align}
Similar calculations show that
\begin{align}
\nonumber 2\skew\big(\mathcal{K}_{\overline{Q}_{e,h}^\natural}^{c^*}\big)n_0&=2\Big(\skew(\mathcal{K}_{\overline{Q}_{e,h}^\natural})+\skew((0|0|{c^*})[(\D  _x\Theta)^\natural (\eta_3)]^{-1})\Big)n_0\\
&=-\mathcal{K}_{\overline{Q}_{e,h}^\natural}^Tn_0+{c^*}-((0|0|{c^*})[(\D  _x\Theta)^\natural (\eta_3)]^{-1})^Tn_0\,,
\end{align}
while the trace term is calculated to be
\begin{align}
\nonumber 2\,\tr(\mathcal{K}_{\overline{Q}_{e,h}^\natural}^{c^*})n_0&=2\Big(\tr(\mathcal{K}_{\overline{Q}_{e,h}^\natural})+\tr((0|0|{c^*})[(\D  _x\Theta)^\natural (\eta_3)]^{-1})\Big)n_0\\&=2\,\tr(\mathcal{K}_{\overline{Q}_{e,h}^\natural})n_0+2\iprod{(0|0|{c^*})[(\D  _x\Theta)^\natural (\eta_3)]^{-1},\id_3}_{\R^{3\times 3}}\,n_0\notag\\
&=2\,\tr(\mathcal{K}_{\overline{Q}_{e,h}^\natural})n_0+2\iprod{{c^*},\underbrace{[(\D  _x\Theta)^\natural(\eta_3) ]^{-T}e_3}_{=n_0}}n_0=2\,\tr(\mathcal{K}_{\overline{Q}_{e,h}^\natural})n_0+2\, {c^*}\,n_0\otimes n_0\,.
\end{align}
By using (\ref{symskewtra curv}), we obtain
\begin{align}
\nonumber	&b_1\mathcal{K}_{\overline{Q}_{e,h}^\natural}^Tn_0+b_1{c^*}+b_1((0|0|{c^*})[(\D  _x\Theta)^\natural(\eta_3) ]^{-1})^Tn_0-b_2\mathcal{K}_{\overline{Q}_{e,h}^\natural}^Tn_0+b_2{c^*}\\
&\qquad-b_2((0|0|{c^*})[(\D  _x\Theta)^\natural (\eta_3)]^{-1})^Tn_0+2b_3\tr(\mathcal{K}_{\overline{Q}_{e,h}^\natural})n_0+2b_3\, {c^*}\,n_0\otimes n_0=0\,.
\end{align}
Gathering similar terms gives us
\begin{align}
\nonumber(b_1-b_2)\mathcal{K}_{\overline{Q}_{e,h}^\natural}^Tn_0+(b_1+b_2){c^*}+(b_1&-b_2)((0|0|{c^*})[(\D  _x\Theta)^\natural(\eta_3) ]^{-1})^Tn_0\\
&+2b_3\tr(\mathcal{K}_{\overline{Q}_{e,h}^\natural})n_0+2b_3\, {c^*}\,n_0\otimes n_0=0\,.
\end{align}
We have
\begin{align}
\nonumber((0|0|{c^*})[(\D  _x\Theta)^\natural(\eta_3) ]^{-1})^Tn_0&=({c^*}\,(0|0|e_3)[(\D  _x\Theta)^\natural (\eta_3)]^{-1})^Tn_0=({c^*}\,n_0)^Tn_0\\
&=n_0^T{c^*}^Tn_0=\iprod{n_0,{c^*}^T}n_0=n_0\iprod{n_0,{c^*}}=n_0\otimes n_0\,{c^*}={c^*}\,n_0\otimes n_0\,,
\end{align}
and by using the decomposition \cite{birsan2019refined,birsan2020derivation,birsan2023}
$
\id_3\,{c^*}=A_{y_0}\,{c^*}+n_0\otimes n_0\,{c^*}\,,
$
we obtain
\begin{align}
\nonumber(b_1-b_2)\mathcal{K}_{\overline{Q}_{e,h}^\natural}^Tn_0+(b_1+b_2)(A_{y_0}\,{c^*}+n_0\otimes n_0\,{c^*})+(b_1-b_2)n_0\otimes n_0\,{c^*}&\\
+2b_3\tr(\mathcal{K}_{\overline{Q}_{e,h}^\natural})n_0+2b_3\, n_0\otimes n_0\,{c^*}&=0\,,
\end{align}
and
\begin{align}
[(b_1+b_2)A_{y_0}+2(b_1+b_3)n_0\otimes n_0]\,{c^*}&=-(b_1-b_2)\mathcal{K}_{\overline{Q}_{e,h}^\natural}^T n_0-2b_3\tr(\mathcal{K}_{\overline{Q}_{e,h}^\natural})\,n_0\,.
\end{align}
Since $A_{y_0}$ is orthogonal to $n_0\otimes n_0$ and $A_{y_0}^2=A_{y_0}$, 
\begin{align}
\left[\frac{1}{b_1+b_2}A_{y_0}+\frac{1}{2(b_1+b_3)}n_0\otimes n_0\right]	[(b_1+b_2)A_{y_0}+2(b_1+b_3)n_0\otimes n_0]=\id_3
\end{align}
 (see \cite{birsan2020derivation}), we have
\begin{align}
	[(b_1+b_2)A_{y_0}+2(b_1+b_3)n_0\otimes n_0]^{-1}=\frac{1}{b_1+b_2}A_{y_0}+\frac{1}{2(b_1+b_3)}n_0\otimes n_0
\end{align}
and we find
\begin{align}
\nonumber {c^*}=(b_2-b_1)\Big[&\frac{1}{b_1+b_2}A_{y_0}+\frac{1}{2(b_1+b_3)}n_0\otimes n_0\Big]\mathcal{K}_{\overline{Q}_{e,h}^\natural}^Tn_0-2b_3\tr(\mathcal{K}_{\overline{Q}_{e,h}^\natural})\Big[\frac{1}{b_1+b_2}A_{y_0}+\frac{1}{2(b_1+b_3)}n_0\otimes n_0\Big]n_0\,.
\end{align}
Because,
\begin{align}
A_{y_0}\mathcal{K}_{\overline{Q}_{e,h}^\natural}^T&=\id_3\mathcal{K}_{\overline{Q}_{e,h}^\natural}^T-n_0\otimes n_0\mathcal{K}_{\overline{Q}_{e,h}^\natural}^T\notag\\&=\mathcal{K}_{\overline{Q}_{e,h}^\natural}^T-(0|0|n_0)(0|0|n_0)^T[(\D  _x\Theta)^\natural (\eta_3)]^{-T}\Big(\mathrm{axl}(\overline{Q}_{e,h}^{\natural,T}\,\partial_{\eta_1} \overline{Q}_{e,h}^\natural)\,|\, \mathrm{axl}(\overline{Q}_{e,h}^{\natural,T}\,\partial_{\eta_2} \overline{Q}_{e,h}^\natural)\,|0\Big)^Tn_0\notag\\
&=\mathcal{K}_{\overline{Q}_{e,h}^\natural}^T-(0|0|n_0)(\underbrace{[(\D  _x\Theta)^\natural (\eta_3)]^{-1}(0|0|n_0)}_{(0|0|e_3)})^T\Big(\mathrm{axl}(\overline{Q}_{e,h}^{\natural,T}\,\partial_{\eta_1} \overline{Q}_{e,h}^\natural)\,|\, \mathrm{axl}(\overline{Q}_{e,h}^{\natural,T}\,\partial_{\eta_2} \overline{Q}_{e,h}^\natural)\,|0\Big)^Tn_0\\&=\mathcal{K}_{\overline{Q}_{e,h}^\natural}^T-(0|0|n_0)(0|0|e_3)^T\Big(\mathrm{axl}(\overline{Q}_{e,h}^{\natural,T}\,\partial_{\eta_1} \overline{Q}_{e,h}^\natural)\,|\, \mathrm{axl}(\overline{Q}_{e,h}^{\natural,T}\,\partial_{\eta_2} \overline{Q}_{e,h}^\natural)\,|0\Big)^Tn_0\notag
\\&=\mathcal{K}_{\overline{Q}_{e,h}^\natural}^T-(0|0|n_0)\begin{pmatrix}
0&0&0\\
0&0&0\\
0&0&1
\end{pmatrix}\begin{pmatrix}
*&*&*\\
*&*&*\\
0&0&0
\end{pmatrix}n_0=\mathcal{K}_{\overline{Q}_{e,h}^\natural}^T\,,\notag
\end{align}
we obtain the unique minimizer
\begin{align}\label{c}
c^*&=\frac{(b_2-b_1)}{b_1+b_2}\mathcal{K}_{\overline{Q}_{e,h}^\natural}^Tn_0-\frac{2b_3}{2(b_1+b_3)}\tr(\mathcal{K}_{\overline{Q}_{e,h}^\natural})\,n_0\,.
\end{align}

Next, we insert the minimizer $c^*$ in (\ref{c}). We have
\begin{align}\label{reducedcur}
 \norm{\sym \,\mathcal{K}_{\overline{Q}_{e,h}^\natural} ^{c^*}}^2&=\norm{\sym  \big(\mathcal{K}_{\overline{Q}_{e,h}^\natural}\big) }^2+ \norm{\sym\big((0|0|c)[(\D  _x\Theta)^\natural(\eta_3) ]^{-1}\big)}^2\notag\\
\nonumber &\qquad+2\dyniprod{\sym  \big(\mathcal{K}_{\overline{Q}_{e,h}^\natural}\big) ,\sym\big((0|0|c)[(\D  _x\Theta)^\natural(\eta_3) ]^{-1}\big)}\\
\nonumber &= \norm{\sym  (\mathcal{K}_{\overline{Q}_{e,h}^\natural} ) }^2\\&\quad +\norm{\sym\Big(\frac{b_2-b_1}{b_1+b_2} \mathcal{K}_{\overline{Q}_{e,h}^\natural} ^T(0|0|n_0)[(\D  _x\Theta)^\natural ]^{-1}-\frac{b_3}{(b_1+b_3)}\tr(  \mathcal{K}_{\overline{Q}_{e,h}^\natural} )(0|0|n_0)[(\D  _x\Theta)^\natural (\eta_3)]^{-1}\Big)}^2\nonumber\\
&\qquad+2\Big\langle\sym  (\mathcal{K}_{\overline{Q}_{e,h}^\natural} ) ,\sym\Big(\frac{b_2-b_1}{b_1+b_2} \mathcal{K}_{\overline{Q}_{e,h}^\natural} ^T(0|0|n_0)[(\D  _x\Theta)^\natural (\eta_3)]^{-1}\\&\hspace{3.5cm}-\frac{b_3}{(b_1+b_3)}\tr(  \mathcal{K}_{\overline{Q}_{e,h}^\natural} )(0|0|n_0)[(\D  _x\Theta)^\natural (\eta_3)]^{-1}\Big)\Big\rangle\,,\notag
\end{align}
and
\begin{align}
\nonumber &\norm{\sym\Big(\frac{b_2-b_1}{b_1+b_2} \mathcal{K}_{\overline{Q}_{e,h}^\natural} ^T(0|0|n_0)[(\D  _x\Theta)^\natural (\eta_3)]^{-1}-\frac{b_3}{(b_1+b_3)}\tr(  \mathcal{K}_{\overline{Q}_{e,h}^\natural} )(0|0|n_0)[(\D  _x\Theta)^\natural (\eta_3)]^{-1}\Big)}^2\\
\nonumber &\qquad =\frac{(b_2-b_1)^2}{(b_1+b_2)^2}\norm{\sym(  \mathcal{K}_{\overline{Q}_{e,h}^\natural} ^T n_0\otimes n_0)}^2+\frac{b_3^2}{(b_1+b_3)^2}\tr(\mathcal{K}_{\overline{Q}_{e,h}^\natural}  )^2\norm{n_0\otimes n_0}^2\\
\nonumber &\quad\qquad-2\,\frac{b_2-b_1}{b_1+b_2}\;\frac{b_3}{(b_1+b_3)}\tr( \mathcal{K}_{\overline{Q}_{e,h}^\natural}  )\iprod{\sym(\mathcal{K}_{\overline{Q}_{e,h}^\natural} ^T n_0\otimes n_0),n_0\otimes n_0}\\
\nonumber &\qquad= \frac{(b_2-b_1)^2}{(b_1+b_2)^2}\dyniprod{\sym(  \mathcal{K}_{\overline{Q}_{e,h}^\natural} ^T n_0\otimes n_0),\sym(  \mathcal{K}_{\overline{Q}_{e,h}^\natural} ^T n_0\otimes n_0)}+\frac{b_3^2}{(b_1+b_3)^2}\tr(  \mathcal{K}_{\overline{Q}_{e,h}^\natural}  )^2\\\notag
&\quad\qquad-\frac{b_2-b_1}{b_1+b_2}\;\frac{b_3}{(b_1+b_3)}\tr(\mathcal{K}_{\overline{Q}_{e,h}^\natural}  )\iprod{\mathcal{K}_{\overline{Q}_{e,h}^\natural} ^T n_0\otimes n_0,n_0\otimes n_0}\\&\quad\qquad -\frac{b_2-b_1\,}{b_1+b_2}\;\frac{b_3}{(b_1+b_3)}\tr(  \mathcal{K}_{\overline{Q}_{e,h}^\natural} )\iprod{n_0\otimes n_0\,  \mathcal{K}_{\overline{Q}_{e,h}^\natural}  ,n_0\otimes n_0}\\
\nonumber &\qquad= \frac{(b_2-b_1)^2}{4(b_1+b_2)^2}\iprod{\mathcal{K}_{\overline{Q}_{e,h}^\natural} ^T n_0\otimes n_0,\mathcal{K}_{\overline{Q}_{e,h}^\natural} ^T n_0\otimes n_0}+\frac{(b_2-b_1)^2}{4(b_1+b_2)^2}\iprod{\mathcal{K}_{\overline{Q}_{e,h}^\natural} ^T n_0\otimes n_0,n_0\otimes n_0 \, \mathcal{K}_{\overline{Q}_{e,h}^\natural} }\\
\nonumber &\quad\qquad+\frac{(b_2-b_1)^2}{4(b_1+b_2)^2}\iprod{n_0\otimes n_0\,  \mathcal{K}_{\overline{Q}_{e,h}^\natural}   ,  \mathcal{K}_{\overline{Q}_{e,h}^\natural} ^T n_0\otimes n_0}+\frac{(b_2-b_1)^2}{4(b_1+b_2)^2}\iprod{n_0\otimes n_0\,  \mathcal{K}_{\overline{Q}_{e,h}^\natural}   ,n_0\otimes n_0\,  \mathcal{K}_{\overline{Q}_{e,h}^\natural}  }\\
\nonumber &\quad\qquad+\frac{b_3^2}{(b_1+b_3)^2}\tr(\mathcal{K}_{\overline{Q}_{e,h}^\natural}  )^2=\frac{(b_2-b_1)^2}{2(b_1+b_2)^2}\norm{\mathcal{K}_{\overline{Q}_{e,h}^\natural} ^T n_0}^2+\frac{b_3^2}{(b_1+b_3)^2}\tr(\mathcal{K}_{\overline{Q}_{e,h}^\natural}  )^2.
\end{align}
Note that
\begin{align}
\nonumber& \iprod{\mathcal{K}_{\overline{Q}_{e,h}^\natural}  ^T,n_0\otimes n_0}=\iprod{\mathcal{K}_{\overline{Q}_{e,h}^\natural}^T,n_0\otimes n_0}\\&=\iprod{\Big(\mathrm{axl}(\overline{Q}_{e,h}^{\natural,T}\,\partial_{\eta_1} \overline{Q}_{e,h}^\natural)\,|\, \mathrm{axl}(\overline{Q}_{e,h}^{\natural,T}\,\partial_{\eta_2} \overline{Q}_{e,h}^\natural)\,|0\Big)[(\D  _x\Theta)^\natural(\eta_3) ]^{-1},(0|0|n_0)[(\D  _x\Theta)^\natural(0) ]^{-1}}\notag\\
\nonumber &=\dyniprod{\Big(\mathrm{axl}(\overline{Q}_{e,h}^{\natural,T}\,\partial_{\eta_1} \overline{Q}_{e,h}^\natural)\,|\, \mathrm{axl}(\overline{Q}_{e,h}^{\natural,T}\,\partial_{\eta_2} \overline{Q}_{e,h}^\natural)\,|0\Big),(0|0|n_0)[(\D  _x\Theta)^\natural(0) ]^{-1}[(\D  _x\Theta)^\natural(0) ]^{-T}\matr{&\id_2-x_3L_{y_0}&0\\ & & 0\\0&0&1}^{-T}}\\
\nonumber &=\dyniprod{(0|0|n_0)^T\Big(\mathrm{axl}(\overline{Q}_{e,h}^{\natural,T}\,\partial_{\eta_1} \overline{Q}_{e,h}^\natural)\,|\, \mathrm{axl}(\overline{Q}_{e,h}^{\natural,T}\,\partial_{\eta_2} \overline{Q}_{e,h}^\natural)\,|0\Big), \matr{&\text{I}_{y_0}^{-1}&0\\ & & 0\\0&0&1}\matr{&\id_2-x_3L_{y_0}&0\\ & & 0\\0&0&1}^{-T}}\\
&=\dyniprod{\matr{0&0&0\\0&0&0\\ * & * &0},\matr{ * & * &0\\ * & * &0\\0&0&1}}=0\,,
\end{align}
where
the {\it Weingarten map (or shape operator)} is defined by 
$
{\rm L}_{y_0}\,=\, {\rm I}_{y_0}^{-1} {\rm II}_{y_0}\in \mathbb{R}^{2\times 2},
$
where ${\rm I}_{y_0}:=[{\D   y_0}]^T\,{\D   y_0}\in \mathbb{R}^{2\times 2}$ and  ${\rm II}_{y_0}:\,=\,-[{\D   y_0}]^T\,{\D   n_0}\in \mathbb{R}^{2\times 2}$ are  the matrix representations of the {\it first fundamental form (metric)} and the  {\it  second fundamental form} of the surface, respectively.  
We also observe that
\begin{align}
\nonumber n_0\otimes n_0[(\D  _x\Theta)^\natural(\eta_3) ]^{-T}&=(0|0|n_0)[(\D  _x\Theta)^\natural(0) ]^{-1}[(\D  _x\Theta)^\natural(\eta_3) ]^{-T}\\
&=(0|0|n_0)[(\D  _x\Theta)^\natural(0) ]^{-1}[(\D  _x\Theta)^\natural(0) ]^{-T}\matr{& \id_2-x_3L_{y_0}&0\\ & & 0\\ 0&0&1}^{-T}\\
&=(0|0|n_0)\matr{&\text{I}_{y_0}^{-1}&0\\& &0\\0&0&1}\matr{& \id_2-x_3L_{y_0}&0\\ & & 0\\ 0&0&1}^{-T}=(0|0|n_0)\matr{ * & * & 0\\ * & * & 0\\0&0&1}=(0|0|n_0)\notag\,.
\end{align}
For every vector $\widehat{u},v\in \R^3$ we have
\begin{align}
\nonumber\iprod{\widehat{u}\otimes n_0,v\otimes n_0}&=\iprod{(v\otimes n_0)^T\widehat{u}\otimes n_0,\id}=\iprod{(n_0\otimes v)\widehat{u}\otimes n_0,\id}=\iprod{n_0\otimes n_0\iprod{v,\widehat{u}},\id}=\iprod{v,\widehat{u}}\cdot \underbrace{\iprod{n_0,n_0}}_{=1}=\iprod{v,\widehat{u}}\,,
\end{align}
and $n_0\otimes n_0=(0|0|n_0)[(\D  _x\Theta)^\natural(0) ]^{-1}$, we deduce
\begin{align}\label{transpose Emscurv}
\iprod{  \mathcal{K}_{\overline{Q}_{e,h}^\natural} ^T n_0\otimes n_0,  \mathcal{K}_{\overline{Q}_{e,h}^\natural} ^T n_0\otimes n_0}=\iprod{  \mathcal{K}_{\overline{Q}_{e,h}^\natural} ^T n_0,  \mathcal{K}_{\overline{Q}_{e,h}^\natural} ^T n_0}=\norm{  \mathcal{K}_{\overline{Q}_{e,h}^\natural} ^T n_0}^2\,.
\end{align}

On the other hand,
\begin{align}
2\dyniprod{\sym  \mathcal{K}_{\overline{Q}_{e,h}^\natural}  ,\sym(\frac{b_2-b_1}{b_1+b_2}\mathcal{K}_{\overline{Q}_{e,h}^\natural} ^Tn_0\otimes n_0-\frac{b_3}{(b_1+b_3)}\tr(\mathcal{K}_{\overline{Q}_{e,h}^\natural} )n_0\otimes n_0)}=\frac{b_2-b_1}{b_1+b_2}\norm{\mathcal{K}_{\overline{Q}_{e,h}^\natural} ^T n_0}^2\,.
\end{align}
Therefore
\begin{align}\label{symtercurv}
\norm{\sym \mathcal{K}_{\overline{Q}_{e,h}^\natural}^{c^*}}^2=\norm{\sym  \mathcal{K}_{\overline{Q}_{e,h}^\natural}  }^2+\frac{(b_1-b_2)^2}{2(b_1+b_2)^2}\norm{\mathcal{K}_{\overline{Q}_{e,h}^\natural} ^T n_0}^2+\frac{b_3^2}{(b_1+b_3)^2}\tr(\mathcal{K}_{\overline{Q}_{e,h}^\natural}  )^2+\frac{b_2-b_1}{b_1+b_2}\norm{\mathcal{K}_{\overline{Q}_{e,h}^\natural} ^T n_0}^2\,.
\end{align}

Now we continue the calculations for the skew-symmetric part,
\begin{align}\label{skewapcurv}
\norm{\skew\mathcal{K}_{\overline{Q}_{e,h}^\natural}^{c^*}}^2=&\norm{\skew \mathcal{K}_{\overline{Q}_{e,h}^\natural} }^2 + \norm{\skew((0|0|c^*)[(\D  _x\Theta)^\natural (\eta_3)]^{-1})}^2\notag\\&+2\iprod{\skew\mathcal{K}_{\overline{Q}_{e,h}^\natural} , \skew((0|0|c^*)[(\D  _x\Theta)^\natural(\eta_3) ]^{-1})}.
\end{align}
In a similar manner, we calculate the terms separately. Since $n_0\otimes n_0$ is symmetric, we obtain
\begin{align}
\norm{\skew((0|0|c)[(\D  _x\Theta)^\natural (\eta_3)]^{-1})}^2&=\norm{\skew(\frac{b_2-b_1}{b_1+b_2}  \mathcal{K}_{\overline{Q}_{e,h}^\natural} ^T \, n_0\otimes n_0-\frac{b_3}{(b_1+b_3)}\tr(  \mathcal{K}_{\overline{Q}_{e,h}^\natural} )\, n_0\otimes n_0)}^2\\
\nonumber&=\frac{(b_1-b_2)^2}{(b_1+b_2)^2}\norm{\skew(  \mathcal{K}_{\overline{Q}_{e,h}^\natural} ^T\, n_0\otimes n_0)}^2.
\end{align}
Using that $(n_0\otimes n_0)^2=(n_0\otimes n_0)$ we deduce
\begin{align}
\norm{\skew(\mathcal{K}_{\overline{Q}_{e,h}^\natural} ^T n_0\otimes n_0)}^2&=\frac{1}{4}\dyniprod{  \mathcal{K}_{\overline{Q}_{e,h}^\natural} ^T n_0\otimes n_0,  \mathcal{K}_{\overline{Q}_{e,h}^\natural} ^T\, n_0\otimes n_0}-\frac{1}{4}\dyniprod{  \mathcal{K}_{\overline{Q}_{e,h}^\natural} ^T n_0\otimes n_0,n_0\otimes n_0 \, \mathcal{K}_{\overline{Q}_{e,h}^\natural} } \notag\\
&\quad-\frac{1}{4}\dyniprod{n_0\otimes n_0 \, \mathcal{K}_{\overline{Q}_{e,h}^\natural}  ,  \mathcal{K}_{\overline{Q}_{e,h}^\natural} ^T\, n_0\otimes n_0}+\frac{1}{4}\dyniprod{n_0\otimes n_0\,  \mathcal{K}_{\overline{Q}_{e,h}^\natural}  ,n_0\otimes n_0 \, \mathcal{K}_{\overline{Q}_{e,h}^\natural}  }\\
&=\frac{1}{2}\norm{ \mathcal{K}_{\overline{Q}_{e,h}^\natural} ^T n_0}^2\,.\notag
\end{align}
 We have as well
\begin{align}\label{appskewcurv}
2\iprod{&\skew\mathcal{K}_{\overline{Q}_{e,h}^\natural} , \skew((0|0|c^*)[(\D  _x\Theta)^\natural (\eta_3)]^{-1})}= 2\,\frac{(b_2-b_1)}{(b_1+b_2)}\dyniprod{\skew\mathcal{K}_{\overline{Q}_{e,h}^\natural} ,\skew(\mathcal{K}_{\overline{Q}_{e,h}^\natural} ^T n_0\otimes n_0)}\\
\nonumber&=\frac{(b_2-b_1)}{2(b_1+b_2)}\iprod{\mathcal{K}_{\overline{Q}_{e,h}^\natural} ,  \mathcal{K}_{\overline{Q}_{e,h}^\natural} ^T n_0\otimes n_0}-\frac{(b_2-b_1)}{2(b_1+b_2)}\iprod{\mathcal{K}_{\overline{Q}_{e,h}^\natural} ,n_0\otimes n_0\,  \mathcal{K}_{\overline{Q}_{e,h}^\natural} }\\
\nonumber&\quad-\frac{(b_2-b_1)}{2(b_1+b_2)}\iprod{\mathcal{K}_{\overline{Q}_{e,h}^\natural} ^T,  \mathcal{K}_{\overline{Q}_{e,h}^\natural} ^T n_0\otimes n_0}+\frac{(b_2-b_1)}{2(b_1+b_2)}\iprod{\mathcal{K}_{\overline{Q}_{e,h}^\natural} ^T,n_0\otimes n_0 \, \mathcal{K}_{\overline{Q}_{e,h}^\natural}  }=-\frac{(b_2-b_1)}{(b_1+b_2)}\norm{  \mathcal{K}_{\overline{Q}_{e,h}^\natural} ^T n_0}^2\,,\notag
\end{align}
and we obtain
\begin{align}\label{skewcurv}
\norm{\skew\mathcal{K}_{\overline{Q}_{e,h}^\natural}^{c^*}}^2&=\norm{\skew\mathcal{K}_{\overline{Q}_{e,h}^\natural} }^2 +\frac{(b_2-b_1)^2}{2(b_1+b_2)^2}\norm{  \mathcal{K}_{\overline{Q}_{e,h}^\natural} ^T n_0}^2-\frac{(b_2-b_1)}{(b_1+b_2)}\norm{  \mathcal{K}_{\overline{Q}_{e,h}^\natural} ^T n_0}^2.
\end{align}
Finally, the trace-term is computed. A further needed calculation is
\begin{align}\label{tracecurv}
\Big[\tr(\mathcal{K}_{\overline{Q}_{e,h}^\natural}^{c^*})\Big]^2&=\Big(\tr(\mathcal{K}_{\overline{Q}_{e,h}^\natural} ) +\tr\big((0|0|c)[(\D  _x\Theta)^\natural (\eta_3)]^{-1})\big)\Big)^2\\
\nonumber&=\Big(\tr(\mathcal{K}_{\overline{Q}_{e,h}^\natural}  )+\frac{(b_2-b_1)}{2(b_1+b_2)}\iprod{\mathcal{K}_{\overline{Q}_{e,h}^\natural} ^T n_0\otimes n_0,\id_3}-\frac{b_3}{(b_1+b_3)}\tr( \mathcal{K}_{\overline{Q}_{e,h}^\natural}  )\underbrace{\iprod{n_0\otimes n_0,\id_3}}_{\iprod{n_0,n_0}=1}\Big)^2\\
&=\frac{b_1^2}{(b_1+b_3)^2}\tr(\mathcal{K}_{\overline{Q}_{e,h}^\natural}  )^2.\notag
\end{align}
Now we insert the above calculations in $\widetilde{W}_{\text{curv}}( \mathcal{K}_{\overline{Q}_{e,h}^\natural}+(0|0|c^*)[(\D  _x\Theta)^\natural(\eta_3)]^{-1})$, and obtain
\begin{align}
\nonumber {W}_{\text{curv}}^{\text{hom}}&=\mu L_c^2\Big(b_1(\norm{\sym\mathcal{K}_{\overline{Q}_{e,h}^\natural} }^2+\frac{(b_1-b_2)^2}{2(b_1+b_2)^2}\norm{\mathcal{K}_{\overline{Q}_{e,h}^\natural} ^Tn_0}^2+\frac{b_3^2}{(b_1+b_3)^2}\tr(\mathcal{K}_{\overline{Q}_{e,h}^\natural} )^2+\frac{b_2-b_1}{b_1+b_2}\norm{\mathcal{K}_{\overline{Q}_{e,h}^\natural} ^Tn_0}^2)\\
\nonumber&\qquad\qquad+b_2(\norm{\skew \mathcal{K}_{\overline{Q}_{e,h}^\natural} }^2+\frac{(b_2-b_1)^2}{2(b_1-b_2)^2}\norm{\mathcal{K}_{\overline{Q}_{e,h}^\natural} ^Tn_0}^2-\frac{b_2-b_1}{b_1+b_2}\norm{\mathcal{K}_{\overline{Q}_{e,h}^\natural} ^Tn_0}^2)\\
&\qquad\qquad +b_3\frac{b_1^2}{(b_1+b_3)^2}\tr(\mathcal{K}_{\overline{Q}_{e,h}^\natural} )^2\Big)\,,
\end{align}
which reduces to
\begin{align}
W_{\text{curv}}^{\text{hom}}(\mathcal{K}_{\overline{Q}_{e,h}^\natural} )&=\mu L_c^2\Big(b_1\norm{\sym\mathcal{K}_{\overline{Q}_{e,h}^\natural} }^2+b_2\norm{\skew \mathcal{K}_{\overline{Q}_{e,h}^\natural} }^2-\frac{(b_1-b_2)^2}{2(b_1+b_2)}\norm{\mathcal{K}_{\overline{Q}_{e,h}^\natural} ^Tn_0}^2+\frac{b_1b_3}{(b_1+b_3)}\tr(\mathcal{K}_{\overline{Q}_{e,h}^\natural} )^2\Big)\,.
\end{align}
One may apply the orthogonal decomposition of a matrix $X$  
\begin{align}\label{orthog decom}
X=X^\parallel+X^\perp, \qquad \qquad X^\parallel\coloneqq{\rm A}_{y_0} \,X,  \qquad\qquad X^\perp\coloneqq(\id_3-{\rm A}_{y_0}) \,X,
\end{align}
for  the matrix $\mathcal{K}_{\overline{Q}_{e,h}^\natural} $, where $A_{y_0}=(\D  y_0|0)[\D  _x\Theta(0)]^{-1}$. After inserting the decomposition in the homogenized curvature energy, we get
\begin{align}\label{homocurv}
\nonumber	W_{\text{curv}}^{\text{hom}}(\mathcal{K}_{\overline{Q}_{e,h}^\natural} )&=\mu L_c^2\Big(b_1\norm{\sym\mathcal{K}_{\overline{Q}_{e,h}^\natural} }^2+b_2\norm{\skew \mathcal{K}_{\overline{Q}_{e,h}^\natural} }^2-\frac{(b_1-b_2)^2}{2(b_1+b_2)}\norm{\mathcal{K}_{\overline{Q}_{e,h}^\natural} ^Tn_0}^2+\frac{b_1b_3}{(b_1+b_3)}\tr(\mathcal{K}_{\overline{Q}_{e,h}^\natural} )^2\Big)\\
&=\mu L_c^2\Big(b_1\norm{\sym\mathcal{K}_{\overline{Q}_{e,h}^\natural} ^\parallel}^2+b_2\norm{\skew \mathcal{K}_{\overline{Q}_{e,h}^\natural} ^\parallel}^2-\frac{(b_1-b_2)^2}{2(b_1+b_2)}\norm{\mathcal{K}_{\overline{Q}_{e,h}^\natural} ^Tn_0}^2\\&\qquad+\frac{b_1b_3}{(b_1+b_3)}\tr(\mathcal{K}_{\overline{Q}_{e,h}^\natural} ^\parallel)^2+\frac{b_1+b_2}{2}\norm{\mathcal{K}_{\overline{Q}_{e,h}^\natural} ^Tn_0}\Big)\nonumber\\
&=\mu L_c^2\Big(b_1\norm{\sym\mathcal{K}_{\overline{Q}_{e,h}^\natural} ^\parallel}^2+b_2\norm{\skew \mathcal{K}_{\overline{Q}_{e,h}^\natural} ^\parallel}^2+\frac{b_1b_3}{(b_1+b_3)}\tr(\mathcal{K}_{\overline{Q}_{e,h}^\natural} ^\parallel)^2+\frac{2b_1b_2}{b_1+b_2}\norm{\mathcal{K}_{\overline{Q}_{e,h}^\natural} ^\perp}\Big)\,.\hspace{1cm}\qedhere\notag
\end{align}

\end{proof}

As regards the homogenized curvature energy for the following curvature energy given by the optimization problem O4, some  simplified calculations as for the optimization problem O3 lead us to

\begin{theorem}The solution of the optimization problem O3 given in \eqref{inf curv} is given by
	\begin{align}c^*&=\frac{(b_2-b_1)}{b_1+b_2}\mathcal{K}_{\overline{Q}_{e,0}}^Tn_0-\frac{2b_3}{2(b_1+b_3)}\tr(\mathcal{K}_{\overline{Q}_{e,0}})n_0\end{align}
	and the coresponding homogenized curvature energy is
\begin{align}\label{homocurv0}
W_{\text{curv}}^{\text{hom}}(\mathcal{K}_{\overline{Q}_{e,0} } )
&=\mu L_c^2\Big(b_1\norm{\sym\mathcal{K}_{\overline{Q}_{e,0} } ^\parallel}^2+b_2\norm{\skew \mathcal{K}_{\overline{Q}_{e,0} } ^\parallel}^2+\frac{b_1b_3}{(b_1+b_3)}\tr(\mathcal{K}_{\overline{Q}_{e,0} } ^\parallel)^2+\frac{2b_1b_2}{b_1+b_2}\norm{\mathcal{K}_{\overline{Q}_{e,0} } ^\perp}\Big)\,.
\end{align}
where $\mathcal{K}_{\overline{Q}_{e,0}} ^\parallel$ and $\mathcal{K}_{\overline{Q}_{e,0}} ^\perp$ represent the orthogonal decomposition of 
	the  fully reduced elastic shell bending-curvature tensor $\mathcal{K}_{\overline{Q}_{e,0}}$ in the tangential plane and in the normal direction, respectively.
\end{theorem}

\subsection{$\Gamma$-convergence result for the curved shell model}
Together with the calculations provided in \cite{Maryam}, we obtain for the first time in the literature the explicit form of the Cosserat-shell model via $\Gamma$-convergence method given by the following theorem.
\begin{theorem}\label{Theorem homo}
	Assume  that the initial configuration of the curved shell is defined by  a continuous injective mapping $\,y_0:\omega\subset\mathbb{R}^2\rightarrow\mathbb{R}^3$  which admits an extension to $\overline{\omega}$ into  $C^2(\overline{\omega};\mathbb{R}^3)$ such that for $$\Theta(x_1,x_2,x_3)=y_0(x_1,x_2)+x_3\, n_0(x_1,x_2)$$ we have $\det[\D _x\Theta(0)] \geq\, a_0 >0$ on $\overline{\omega}$,
	where $a_0$ is a constant, and  assume that the boundary data satisfy the conditions
	\begin{equation}\label{2n5}
	\varphi^\natural_d=\varphi_d\big|_{\Gamma_1} \text{(in the sense of traces) for } \ \varphi_d\in {\rm H}^1(\Omega_1;\mathbb{R}^3).
	\end{equation}
	Let the constitutive parameters satisfy 
	\begin{align}
	\mu\,>0, \qquad\quad \kappa>0, \qquad\quad \mu_{\rm c}> 0,\qquad\quad a_1>0,\quad\quad a_2>0,\qquad\quad a_3>0\,.
	\end{align} 
	Then, 
	for any sequence $(\varphi_{h_j}^\natural, \overline{Q}_{e,h_j}^{\natural})\in X$ such that $(\varphi_{h_j}^\natural, \overline{Q}_{e,h_j}^{\natural})\rightarrow(\varphi_0, \overline{Q}_{e,0})$ as $h_j\to 0$,	the sequence of functionals $\mathcal{J}_{h_j}\col X\rightarrow \overline{\mathbb{R}}$ from \eqref{energyinnewarea} 
	\textit{ $\Gamma$-converges} to the limit energy  functional $\mathcal{J}_0\col X\rightarrow \overline{\mathbb{R}}$ defined by
	\begin{equation}\label{couple}
	\mathcal{J}_0 (m,\overline{Q}_{e,0}) =
	\begin{cases}\dd
	\int_{\omega} [ {W}_{\rm mp}^{{\rm hom}}(\mathcal{E}_{m,\overline{Q}_{e,0}})+\widetilde{W}_{\rm curv}^{\rm hom}(\mathcal{K}_{\overline{Q}_{e,0}})]\det (\D  y_0|n_0)\; d\omega\quad &\text{if}\quad (m,\overline{Q}_{e,0})\in \mathcal{S}_\omega' \,,
	\\
	+\infty\qquad &\text{else in}\; X,
	\end{cases}
	\end{equation}
	where 
	\begin{align}\label{approxi}
	m(x_1,x_2)&:=\varphi_0 (x_1,x_2)=\lim_{h_j\to 0} \varphi_{h_j}^\natural(x_1,x_2,\frac{1}{h_j}x_3), \qquad \overline{Q}_{e,0}(x_1,x_2)=\lim_{h_j\to 0} \overline{Q}_{e,h_j}^\natural(x_1,x_2,\frac{1}{h_j}x_3),\notag\\
	\mathcal{E}_{m,\overline{Q}_{e,0}}&=(\overline{Q}^{T}_{e,0}\D  m-\D  y_0|0)[\D _x\Theta(0)]^{-1},\\
	\mathcal{K}_{\overline{Q}_{e,0}} &=\Big(\mathrm{axl}(\overline{Q}^{T}_{e,0}\,\partial_{x_1} \overline{Q}_{e,0})\,|\, \mathrm{axl}(\overline{Q}^{T}_{e,0}\,\partial_{x_2} \overline{Q}_{e,0})\,|0\Big)[\D _x\Theta (0)\,]^{-1}\not\in {\rm Sym}(3)\,,\notag
	\end{align}
	and
	\begin{align}
	\nonumber {W}_{\rm mp}^{\rm hom}
	(\mathcal{E}_{m,\overline{Q}_{e,0}})&=
	\, \mu\,\lVert  \mathrm{sym}\,   \,\mathcal{E}_{m,\overline{Q}_{e,0} }^{\parallel}\rVert^2 +  \mu_{\rm c}\,\lVert \mathrm{skew}\,   \,\mathcal{E}_{m,\overline{Q}_{e,0} }^{\parallel}\rVert^2 +\,\dfrac{\lambda\,\mu\,}{\lambda+2\,\mu\,}\,\big[ \mathrm{tr}    (\mathcal{E}_{m,\overline{Q}_{e,0} }^{\parallel})\big]^2 +\frac{2\,\mu\, \,  \mu_{\rm c}}{\mu_c\,+\mu\,}\norm{\mathcal{E}_{m,\overline{Q}_{e,0} }^Tn_0}^2\\
	&=W_{\mathrm{shell}}\big(   \mathcal{E}_{m,\overline{Q}_{e,0} }^{\parallel} \big)+\frac{2\,\mu\, \,  \mu_{\rm c}}{\mu_c\,+\mu\,}\lVert \mathcal{E}_{m,\overline{Q}_{e,0} }^{\perp}\rVert^2,\\  \widetilde{W}^{{\rm hom}}_{\rm curv}(\mathcal{K}_{\overline{Q}_{e,0}})
	\nonumber&=\inf_{A\in \mathfrak{so}(3)}\widetilde{W}_{\rm curv}\Big(\mathrm{axl}(\overline{Q}^{T}_{e,0}\,\partial_{\eta_1} \overline{Q}_{e,0})\,|\, \mathrm{axl}(\overline{Q}^{T}_{e,0}\,\partial_{\eta_2} \overline{Q}_{e,0})\,|\,\axl(A)\,\Big)[(\D _x \Theta)^\natural(0)]^{-1}\\
	\nonumber&=\mu L_c^2\Big(b_1\norm{\sym\mathcal{K}_{\overline{Q}_{e,0}}^\parallel}^2+b_2\norm{\skew \mathcal{K}_{\overline{Q}_{e,0}}^\parallel}^2+\frac{b_1b_3}{(b_1+b_3)}\tr(\mathcal{K}_{\overline{Q}_{e,0}}^\parallel)^2+\frac{2\,b_1b_2}{b_1+b_2}\norm{\mathcal{K}_{\overline{Q}_{e,0}}^\perp}\Big)\,.
	\end{align}
\end{theorem}

\begin{proof}The proof is completely similar to the proof provided in \cite{Maryam}, where only some implicit properties of the homogenized curvature energy were used and not its explicit form. \end{proof}

\section{Conclusion}\setcounter{equation}{0}

The present paper gives the explicit calculation of the homogenized curvature energy. This explicit form was not directly necessary in order to prove the following Gamma-convergence result, some qualitative properties of $ \widetilde{W}^{{\rm hom},\natural}_{\rm curv}( \mathcal{K}_{\overline{Q}_{e,h}^\natural})$ and $ \widetilde{W}^{{\rm hom}}_{\rm curv}( \mathcal{K}_{\overline{Q}_{e,0} })$ being enough in the proof of  the $\Gamma$-convergence result. 
However, the final $\Gamma$-convergence model has to be written in a explicit form, all the explicit calculations being provided in this paper.  

A comparison between (\ref{homflat}) and (\ref{homcurv}), shows that the homogenized flat curvature energy can thus be obtained from the curved one, and that Theorem \ref{Theorem_homo_plate} may be seen as a corollary of Theorem \ref{Theorem homo}.
Indeed, let us assume that in the homogenized energy which we obtained in (\ref{homocurv}) we have $\D   \Theta=\id_3, \D  y_0=\id$ and $n_0=[\D  _x\Theta(0)]e_3=e_3$, which corresponds to the flat shell case. Then $\overline{Q}_{e,0} =\overline{R}_{0}$,  $\mathcal{K}_{\overline{Q}_{e,0} }=\mathcal{K}_{\overline{R}_{0}}^{\rm plate}$  and
\begin{align}
\mathcal{K}_{\overline{Q}_{e,0} } = \left(\begin{pmatrix}
{\Gamma _\square}& &0\\
&\hspace*{1.5cm}& 0 \\ \hline
\Gamma _{31}&\Gamma _{32}&0
\end{pmatrix}
\right)[(\D  _x\Theta)^\natural ]^{-1}\,=\mathcal{K}_{\overline{R}_{0}}^{\rm plate},\quad \textrm{with}\quad  \Gamma _\square=\matr{\Gamma _{11}&\Gamma _{12}\\\Gamma _{21}&\Gamma _{22}} 
\end{align}
and we have
\begin{align}\label{homcurv}
W_{\text{curv}}^{\text{hom}}(\Gamma )&=\mu L_c^2\Big(b_1\norm{\sym\Gamma _\square}^2+b_2\norm{\skew \Gamma _\square}^2+\frac{b_1b_3}{(b_1+b_3)}\tr(\Gamma _\square)^2+\frac{2b_1b_2}{b_1+b_2}\dynnorm{\matr{\Gamma _{31}\\ \Gamma _{32}}}^2\Big)\,,
\end{align}
and we rediscover the homogenized curvature energy for an initial flat-configuration from  Theorem \ref{Theorem_homo_plate}.

In conclusion, the present paper completes the calculations of the membrane-like model constructed via $\Gamma$-convergence for flat and curved initial configuration of the shell given for the first time in literature the explicit form of the $\Gamma$-limit for both situations.

In \cite{birsan2020derivation}, by using a method which extends the reduction procedure from classical elasticity to  the case of Cosserat shells, B\^{\i}rsan has obtained a Cosserat-shell by considering a general ansatz. For the particular case of a quadratic ansatz for the deformation map and skipping higher order terms, the membrane term of order $O(h)$ from the Birsan's model \cite{birsan2020derivation} coincides with the homogenized membrane energy determined by us in \cite{Maryam}, i.e., in both models the harmonic mean $\displaystyle \frac{2\mu\,\mu_c}{\mu+\mu_c}$of $\mu$ and $\mu_c$ is present. We note that in the model constructed in \cite{GhibaNeffPartI} the algebraic mean of  $\mu$ and $\mu_c$ substitute the role of the  harmonic mean from the model given in \cite{birsan2020derivation} and by the $\Gamma$-convergence model in \cite{Maryam}.

However,  a comparison between  the curvature energy obtained in the current paper as part of the $\Gamma$-limit and  the curvature energy obtained  using other methods \cite{GhibaNeffPartI,birsan2020derivation}, shows that, the weight of the energy term $\norm{\mathcal{K}_{e,s}^\perp}^2
$ are different as following \begin{itemize}
	\item derivation approach \cite{GhibaNeffPartI}  as well as in the model  given in \cite{birsan2020derivation}: the algebraic mean of $b_1$ and $b_2$, i.e., $\displaystyle\frac{b_1 +  b_2}{2}\,$;
	\item $\Gamma$-convergence:  the harmonic mean of $b_1$ and $b_2$, i.e., $\displaystyle\frac{2\,b_1   b_2}{b_1+b_2}\,$.
\end{itemize}

\begin{footnotesize}

\end{footnotesize}
\end{document}